\documentclass[12pt]{amsart}
\usepackage{graphicx}
\usepackage{amsmath}
\usepackage{subfigure}
\usepackage{amssymb} 
\usepackage[usenames,dvipsnames]{color}
\usepackage{pinlabel}
\usepackage{mathtools}
\usepackage{txfonts} 
\usepackage{hyperref}
\hypersetup{linktocpage}

\usepackage{float}
\usepackage{tikz}
\usepackage{imakeidx}
\makeindex

  \usetikzlibrary{hobby}
  \usetikzlibrary{patterns}
    \input xy
  \xyoption{all}

\usepackage{fancyhdr}
\newtheorem{thm}{Theorem}[section]  
\newtheorem{cor}[thm]{Corollary}

\newtheorem{defin}[thm]{Definition} 
\newtheorem{lemma}[thm]{Lemma} 
 
\newtheorem{prop}[thm]{Proposition}

\newtheorem*{defin*}{Definition}
\newcommand{\aaa}{\mbox{$\alpha$}}
\newcommand{\bbb}{\mbox{$\beta$}}

\newcommand{\eee}{\mbox{$\epsilon$}}

\newcommand{\calB}{\mathcal B}

\newcommand{\calG}{\mathcal G}
\newcommand{\calP}{\mathcal P}
\newcommand{\calH}{\mathcal H}

\newcommand{\bs}{\backslash}
\newcommand{\Rrr}{\mathbb R}

\newcommand{\bdd}{\mbox{$\partial$}}
\newcommand{\inter}{\mbox{${\rm int}$}}

\newcommand{\Diff}{\operatorname{Diff}}
\newcommand{\Dih}{\operatorname{Dih}}

\def\zed{{\mathbb Z}}
\def\real{{\mathbb R}}

\newcommand{\frb}{\mathfrak{b}}

\newcommand{\frc}{\mbox{$\mathfrak{c}$}}

\begin{document}  

\title{Powell's Conjecture on the Goeritz group of $S^3$ is stably true}   

\author{Martin Scharlemann}
\address{\hskip-\parindent
        Martin Scharlemann\\
        Mathematics Department\\
        University of California\\
        Santa Barbara, CA 93106-3080 USA}
\email{mgscharl@math.ucsb.edu}

\thanks{Daryl Cooper provided comments and suggestions that greatly clarified Section \ref{sect:dih}.}

\date{\today}

\begin{abstract}  In 1980 J. Powell \cite{Po} proposed that, for every genus $g$, five specific elements suffice to generate the Goeritz group $\calG_g$ of genus $g$ Heegaard splittings of $S^3$.  Powell's Conjecture remains undecided for $g \geq 4$.  Let $\calP_g \subset \calG_g$ denote the subgroup generated by Powell's elements.  Here we show that, for each genus $g$, the natural function $\calG_g  \to 
\calG_{g+1}/\calP_{g+1}$ is trivial.  
\end{abstract}

\maketitle



\section{Introduction} \label{sect:intro}

Following early work of Goeritz \cite{Go}, the {\em genus $g$ Goeritz group} $\calG_g$ of the $3$-sphere can be described as the isotopy classes of orientation-preserving homeomorphisms of the $3$-sphere that leave the standard genus $g$ Heegaard surface $T_g$ invariant.  Goeritz identified a finite set of generators
for $\calG_2$. 
In 1980 J. Powell \cite{Po} extended Goeritz' set of generators to a set of five elements that he believed would generate the Goeritz group for any fixed higher-genus splitting, but his proof contained a serious gap.  In \cite{FS1} the Powell  Conjecture (as it is now called) was confirmed for $\calG_3$ and in \cite{Sc1} it is pointed out that one of Powell's proposed generators is redundant, so in fact only 4 of Powell's proposed generators need to be considered.

Powell's view of the Goeritz group, which we will adopt, is framed somewhat differently.  Following Johnson-McCullough \cite{JM} (who extend the notion to arbitrary compact orientable manifolds), consider the space of left cosets $\Diff(S^3)/\Diff(S^3, T_g)$, where $\Diff(S^3, T_g)$ consists of those orientation-preserving diffeomorphisms of $S^3$ that carry $T_g$ to itself.  The fundamental group $\tilde{\calG}_g = \pi_1(\Diff(S^3)/\Diff(S^3, T_g))$
 of this space projects to the genus $g$ Goeritz group $\calG_g$ (with kernel $\pi_1(SO(3)) = \mathbb{Z}_2$ \cite[p.197]{Po}) as follows:  A non-trivial element is represented by an isotopy of $T_g$ in $S^3$ that begins with the identity and ends with a diffeomorphism of $S^3$ that takes $T_g$ to itself; this diffeomorphism of the pair $(S^3, T_g)$ represents an element of the Goeritz group as defined earlier.  The advantage of this point of view
is that an element of $\tilde{\calG}_g$ can be viewed quite vividly: it is an excursion of $T_g$ in $S^3$ that begins and ends with the standard picture of $T_g \subset S^3$.  Such excursions are how Powell defines his proposed generators in \cite[Figure 4]{Po}.  

\section{Powell's generators re-expressed}

Powell's description of $T_g \subset S^3, g \geq 2$ begins with a round $2$-sphere in $S^3$, to which is then connect-summed  a standard unknotted torus at each of $g$ points in a circumference $\frc$ of the sphere.  These summands $\frb_i, 1 \leq i \leq g$ will be called the standard genus $1$ bubbles; their complement is a $g$-punctured sphere $P_g$. See Figure \ref{fig:Psetup}. For the Heegaard splitting $S^3 = A \cup_{T_g} B$ let $a_i \subset (\frb_i \cap T_g)$ be a circle (unique up to isotopy in the punctured torus $\frb_i \cap T_g$) that bounds a disk $\mu_i$ in the solid torus $A \cap \frb_i$.  Similarly, let $b_i \subset (\frb_i \cap T_g)$ be a circle (unique up to isotopy rel $a_i$ in the punctured torus $\frb_i \cap T_g$)  that bounds a disk $\lambda_i$ in the solid torus $B \cap \frb_i$, chosen so that $|a_i \cap b_i| = 1$.  $a_i$ and $b_i$ will be called respectively the meridian and longitude of the summand $\frb_i$.  Similarly $\mu_i$ and $\lambda_i$ will be called the meridian and longitudinal disks of $\frb_i$.

\begin{figure}[ht!]
\labellist
\small\hair 2pt
\pinlabel  $P_g$ at 140 50
\pinlabel  $\frb_1$ at 310 90
\pinlabel  $\frb_2$ at 260 190
\pinlabel  $\frb_3$ at 140 230
\pinlabel  $\frc$ at 160 200
\pinlabel  $\mu_2$ at 220 220
\pinlabel  $\lambda_2$ at 220 180
\pinlabel  $A$ at 220 125
\pinlabel  $B$ at 260 140
\endlabellist
    \centering
 \includegraphics[scale=0.7]{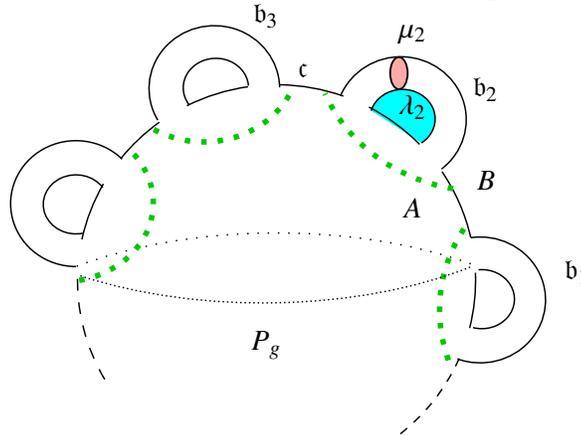}
    \caption{Powell's picture of $T_g$}
    \label{fig:Psetup}
    \end{figure}
    
In terms of this picture, here are Powell's four proposed generators (the fifth being redundant):

\begin{itemize}

\item $D_\omega$ 
is the homeomorphism $(\frb_1, \frb_1 \cap T_g) \to (\frb_1, \frb_1 \cap T_g)$ shown in Figure \ref{fig:flip}.  It preserves both $\mu_1$ and $\lambda_1$ but reverses the orientation of each. We will call this the {\em standard flip} on $\frb_1$.

 \begin{figure}[ht!]
 \labellist
\small\hair 2pt
\pinlabel  $\frb_1$ at 180 50
\pinlabel  $\mu_1$ at 160 110
\pinlabel  $\lambda_1$ at 110 85
\endlabellist
\centering
    \includegraphics[scale=0.6]{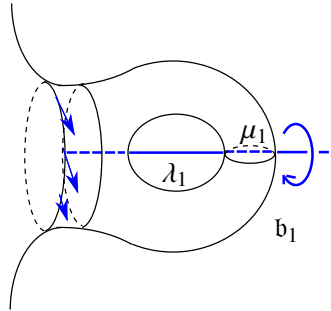}
    \caption{Powell's $D_{\omega}$, the standard flip on $\frb_1$}  \label{fig:flip}
    \end{figure}
    
            \item $D_{\eta}$ is the homeomorphism $(S^3, T_g) \to (S^3, T_g)$ shown in Figure \ref{fig:rotate}: The punctured sphere $P_g$ is rotated by ${2\pi}/{g}$ along the circumference $\frc$ and each standard bubble $\frb_i$ is moved to $\frb_{i+1}$, $i \in \zed/g$, sending each $\mu_i$ to $\mu_{i+1}$ and $\lambda_i$ to $\lambda_{i+1}$.  Fix an orientation for the disks $\mu_1, \lambda_1$ and observe that $(D_{\eta})^{i-1}$ can be used to fix an orientation on each pair $\mu_i, \lambda_i$  that is then preserved by $D_{\eta}$.   
\medskip

 \begin{figure}[ht!]
\centering
    \includegraphics[scale=0.5]{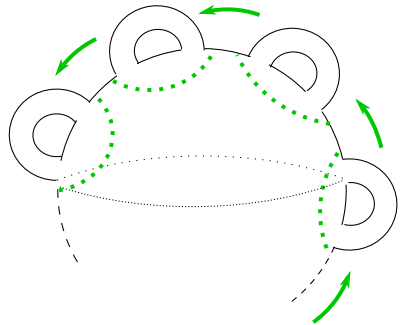}
    \caption{Powell's $D_{\eta}$}  \label{fig:rotate}
    \end{figure}

    \item Let $v \subset P_g$ be the subarc of $\frc$ connecting $\frb_1$ and $\frb_2$. Let $B$ be the reducing ball for the splitting $T_g$ obtained by attaching a $1$-handle regular neighborhood of $v$ to the bubbles $\frb_1, \frb_2$.  $D_{\eta_{12}}$ is the homeomorphism $(B, B \cap T_g) \to (B, B \cap T_g)$ shown in Figure \ref{fig:switch}.  It exchanges the meridian disks $\mu_1$ and $\mu_2$ and the longitudinal disk $\lambda_1$ and $\lambda_2$, preserving orientation of each.   We will call this the {\em standard exchange} of $\frb_1$ and $\frb_2$.  
\medskip

 \begin{figure}[ht!]
  \labellist
\small\hair 2pt
\pinlabel  $\frb_1$ at 310 90
\pinlabel  $\frb_2$ at 270 200
\pinlabel  $v$ at 260 130
\endlabellist
\centering
    \includegraphics[scale=0.5]{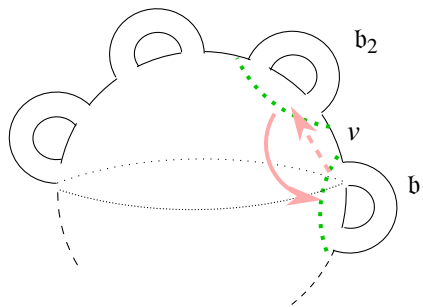}
    \caption{Powell's $D_{\eta_{12}}$, the standard exchange of $\frb_1$ and $\frb_2$}  \label{fig:switch}
    \end{figure}    
    \item 
Let $v \subset T_g$ be an arc connecting $\bdd \mu_1$ to $\bdd \ell_2$, chosen so that inside the bubbles the interior of $v$ is disjoint from the circles $a_1, a_2, b_1, b_2$ and outside the bubbles $v$ is a subarc of $\frc$.  The complex $\mu_1 \cup \ell_2 \cup v$ is called the {\em standard eyeglass} in $T_g$,  see Figure \ref{fig:eyeglass}.  Powell's generator $D_\theta$, illustrated in several panels of \cite[Figure 4]{Po}, makes use of this structure; it can be diagramed more simply as in Figure \ref{fig:eyeglass} and is called the {\em standard eyeglass twist}.
This move will be generalized later, see Figure \ref{fig:eyeglass1}. 

\begin{figure}[ht!]
  \labellist
\small\hair 2pt
\pinlabel  $\frb_1$ at 200 90
\pinlabel  $\frb_2$ at 150 200
\pinlabel  $v$ at 110 100
\endlabellist
    \centering
    \includegraphics[scale=0.7]{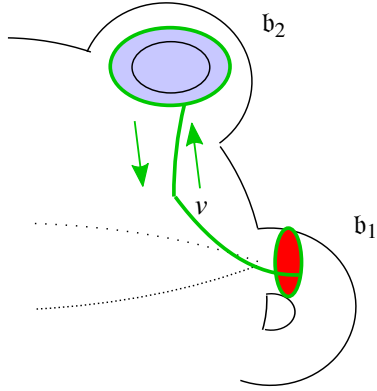}
    \caption{Powell's $D_\theta$, the standard eyeglass twist}  \label{fig:eyeglass}
    \end{figure}
\end{itemize}

As noted earlier, in the exact sequence \[ 1 \to \pi_1(SO(3)) \to \tilde{\calG}_g = \pi_1(\Diff(S^3)/\Diff(S^3, T_g) \to \calG_g \to 1\] Powell's proposed generators actually lie in $\tilde{\calG}_g$, so it makes sense to note the easy:

\begin{prop}  Powell's proposed generators generate $\calG_g$ if and only if they generate $\tilde{\calG}_g$.
\end{prop}

\begin{proof} One direction is obvious; the other follows immediately from the observation that the $2\pi$-rotation $(D_{\eta})^g$ is a Powell move that also represents the non-trivial element of $\pi_1(SO(3))$.
\end{proof}

\begin{defin} 
The Powell group $\calP_g$ is the subgroup of the Goeritz group $\calG_g$ that is generated by Powell's 4 proposed generators $D_{\theta}, D_{\omega}, D_{\eta_{12}}, D_{\eta}$.  Any element of $\calP_g$ (that is, any composition of these generators) will be called a {\em Powell move}.  
\end{defin}

Here is another easy line of argument:  Both of Powell's moves $D_{\eta}$ and $D_{\eta_{12}}$ can be thought of as induced by a homeomorphism on the $g$-punctured sphere $P_g$, and such a homeomorphism also defines an element of the $g$-stranded braid group $\calB_g$ of the sphere.  Furthermore, $(D_{\eta})^{i-1}$ will conjugate the standard exchange $D_{\eta_{12}}$ to a similar exchange (which we will denote $\phi_i$ and also call standard) between $\frb_i$ and $\frb_{i+1}$, so each standard exchange $\phi_i$ is a Powell move.  Moreover, viewed as an element of the braid group $\calB_g$ each exchange $\phi_i, 1 \leq i \leq g-1$ just described is exactly one of the $g-1$ canonical half-twist generators $\sigma_i$ of $\calB_g$ (see \cite[Chapter 9]{FM}). 
Similarly, $(D_{\eta})^{i-1}$ will conjugate $D_{\omega}$ to a flip in the bubble $\frb_i$ that we will call the standard flip in $\frb_i$ and denote $\omega_i \in \calP_g$.  

\begin{prop} \label{prop:newgen}  
The Powell group $\calP_g$ is also generated by these $g+1$ elements of $\calP_g$: \[D_{\theta}, D_{\omega}, \{\phi_i, 1 \leq i \leq g-1\}\]
\end{prop}
\begin{proof}  In the braid group $\calB_g$ the composition $\sigma_1 \sigma_2 ... \sigma_{g-1}$ visibly has the effect of moving $g$ points evenly distributed in the circumference $\frc \subset S^2$ by a rotation of $2\pi/g$.  (See also \cite[Chapter 9]{FM}.)  It follows that for some even $k$ (easily but unimportantly found to be $k = 2$, see below) $D_{\eta} = (D_{\omega})^k \phi_1 \phi_2 ... \phi_{g-1}$ in $\calP_g$.  Thus $D_{\eta}$ lies in the subgroup of $\calP_g$ generated by $D_{\theta}, D_{\omega}, \{\phi_i, 1 \leq i \leq g-1\}$, as does $D_{\eta_{12}} = \phi_1$.  Thus each of the generators $D_{\theta}, D_{\omega}, D_{\eta_{12}}, D_{\eta}$ lies in the subgroup of $\calP_g$ generated by $D_{\theta}, D_{\omega}, \{\phi_i, 1 \leq i \leq g-1\}$, as required.
\end{proof}

Figure \ref{fig:newgen} demonstrates the equality $D_{\eta} = (D_{\omega})^2 \phi_1 \phi_2 ... \phi_{g-1}$ (composed right to left) in the case $g = 4$.  The red arrows in each $\frb_i$ are meant to show orientation within the bubble: the arrows indicate the normal direction to $a_i$.  In the left panel $\phi_1\phi_2\phi_{3}$ rotates the red arrow in $\frb_4$ through an angle $-\frac{3\pi}{2}$ to the red arrow in $\frb_1$; in the right panel $D_{\eta}$ rotates it $\frac{\pi}{2}$.  

\begin{figure}[ht!]
  \labellist
\small\hair 2pt
\pinlabel  $\frb_1$ at 200 30
\pinlabel  $\frb_2$ at 200 200
\pinlabel  $\frb_3$ at 50 200
\pinlabel  $\frb_4$ at 55 25
\pinlabel  $D_{\eta}$ at 410 120
\pinlabel  $\frc$ at 230 120
\pinlabel  $\phi_1\phi_2\phi_{3}$ at 120 120
\endlabellist
    \centering
    \includegraphics[scale=0.6]{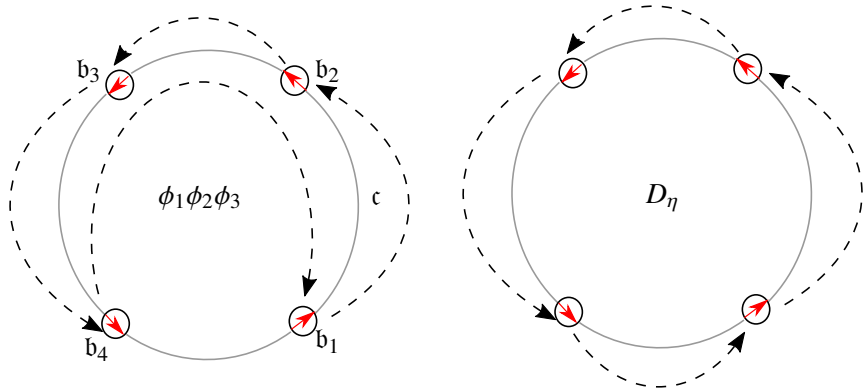}
    \caption{$D_{\eta} = (D_{\omega})^2 \phi_1 \phi_2\phi_{3}$}  \label{fig:newgen}
    \end{figure}
    
    \section{What the Powell group can do}

    We have seen above that the Powell group $\calP_g$ contains any automorphism of $T_g$ obtained by operating on the planar surface $P_g$ by an element of the braid group $\calB_g$. In \cite[Sections 1-4]{FS1} broader types of automorphisms are also shown to lie in $\calP_g$.  For example, \cite[Lemma 1.4]{FS1} says this:
    
   \begin{lemma}   \label{lemma:braid2} 
 Any braid move of a collection of standard bubbles over their complementary surface is a Powell move.   \qed
  \end{lemma}

   Figure \ref{fig:braid2} shows an example in the case of two standard bubbles.
   
 \begin{figure}[ht!]
    \centering
    \includegraphics[scale=0.6]{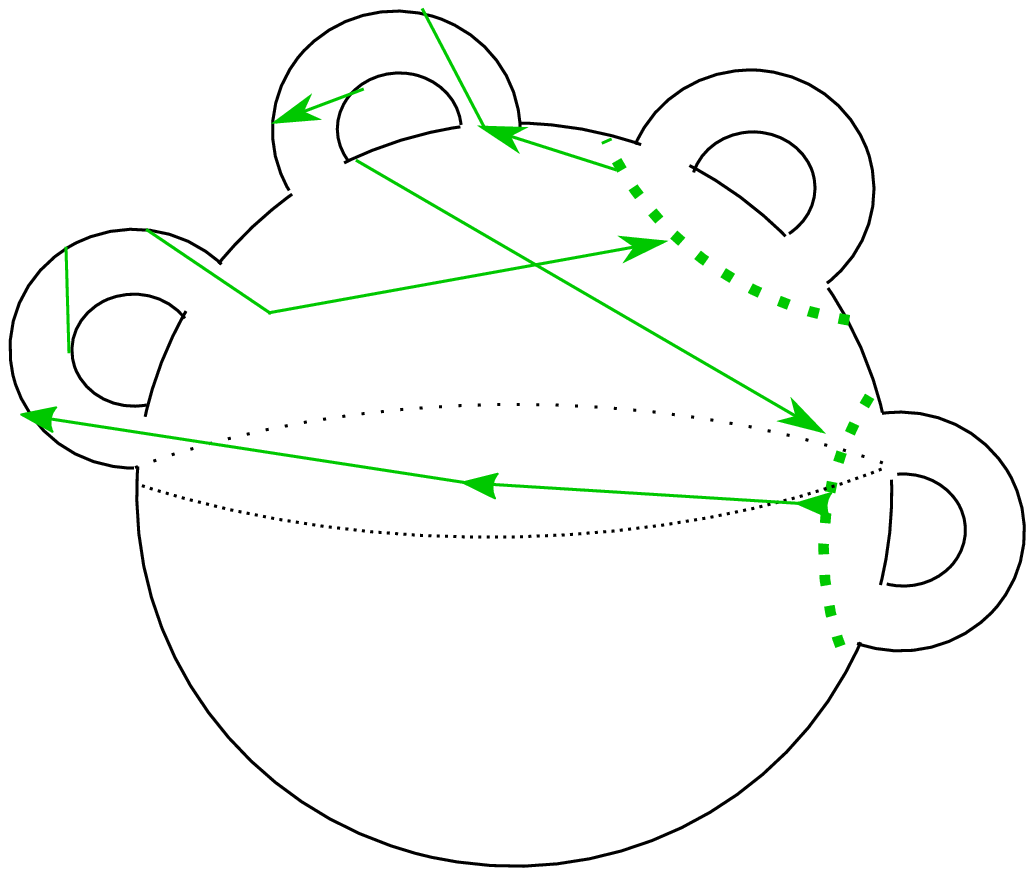}
    \caption{}  \label{fig:braid2}
    \end{figure}
    
    In the case of a single standard bubble $\frb_i$ a braid move is just an isotopy of $\frb_i$ through some path in $T_g - \frb_i$ that returns $\frb_i$ to itself.  This will be called a {\em standard bubble move} of $\frb_i$ and is a Powell move.  More generally, if $\frb$ is an arbitrary bubble, such an isotopy of $\frb_i$ through some path in $T_g - \frb_i$ is called a {\em (generic) bubble move}.  A generic bubble move is not obviously Powell, but see Corollary \ref{cor:eyeglass3o} below.
   \medskip
    
    Here is one consequence:  Let $q: (S^3, T_{g+1}) \to (S^3, T_g)$ be the quotient map that shrinks the standard bubble $\frb_{g+1} \subset S^3$ to a point $\star \in T_g$. Suppose $\tau: (S^3, T_g) \to (S^3, T_g)$ represents an element in $\calG_g$ and let $\alpha$ be an embedded path in $T_g$ from $\tau^{-1}(\star)$ to $\star$.  Let $\alpha_{\tau^{-1}(\star)}^{\star}:(S^3, T_g) \to (S^3, T_g)$ be a homeomorphism whose support lies  in a regular neighborhood of $\alpha$ and pushes $\star$ to $\tau^{-1}(\star)$. (See \cite[Homogeneity Lemma]{Mi}.) Then $\tau\alpha_{\tau^{-1}(\star)}^{\star} (\star) = \star$, so $\tau$ induces $\tau_{\alpha}: (S^3, T_{g+1}) \to (S^3, T_{g+1})$ by replacing a neighborhood of $\star$ with the standard bubble $\frb_{g+1}$. %
\begin{prop} \label{prop:Ginclude}  For any genus $g \geq 2$, the construction $\tau \to \tau_{\alpha}$ described above determines a function $$\iota^+: \calG_g \to 
    \calG_{g+1}/\calP_{g+1}$$ with the property that $\iota^+(\calP_g)$ is trivial.
\end{prop}

\begin{proof}   
First observe that if $\tau_1:(S^3, T_g) \to (S^3, T_g)$ represents the same element in $\calG_g$ as $\tau = \tau_0$ then there is an isotopy $\tau_t, 0 \leq 1$ between them.   Then $\alpha_t = t(\alpha)$ is a path from $\tau_t^{-1}(\star)$ to $\star$ so $(\tau_t)_{\alpha_t}: (S^3, T_{g+1}) \to (S^3, T_{g+1})$ is an isotopy between $(\tau_0)_{\alpha_0}$ and $(\tau_1)_{\alpha_1}$.  Thus they represent the same element of $\calG_{g+1}$. 

Next observe that if $\beta$ is a different path in $T_g$ from $\tau^{-1}(\star)$ to $\star$ then $\tau_{\alpha}$ and $\tau_{\beta}$ differ by a standard bubble move on $\frb_{g+1}$ along the path $\alpha \overline{\beta}$, and by Lemma \ref{lemma:braid2} such a bubble move is in $\calP_{g+1}$.  Thus $\tau_{\alpha}$ and $\tau_{\beta}$ represent the same coset in $\calG_{g+1}/\calP_{g+1}$.
We conclude that the function $\iota^+$ is well-defined.  

To show that $\iota^+(\calP_g)$ lies in $\calP_{g+1}$ consider the generators of $\calP_g$ described in Proposition \ref{prop:newgen}.  Each involves at most two standard bubbles, so the same will be true once $\iota^+$ carries them to $\calG_{g+1}$.  Following Lemma \ref{lemma:braid2} arbitrary exchanges of standard bubbles are Powell moves in $\calG_{g+1}$, so up to Powell moves, we can view the generators of $\iota^+(\calP_g)$ as each involving only the standard bubbles $\frb_1$ and $\frb_2$, where each is a standard generator of $\calP_{g+1}$.  
 \end{proof}

Another useful lemma requires the following generalization of the standard eyeglass move $D_{\theta}$, originally \cite[Definition 2.1]{FS1} :

\begin{defin} \label{defin:eyeglass}
For $S^3 = A \cup_T B$ a Heegaard splitting of $S^3$, an {\em eyeglass} $\eta \subset S^3$ is the union of two properly embedded disjoint disks, $\ell_a \subset A, \ell_b \subset B$ (the {\em lenses} of $\eta$) with an embedded arc $v \subset T - \bdd(\ell_a \cup \ell_b)$ (the {\em bridge} of $\eta$) connecting their boundaries.  The 1-complex $\eta \cap T$ (the bridge, together with the boundary of the lenses) is called the {\em frame} of $\eta$.  

The eyeglass $\eta$ defines a natural automorphism $(S^3, T) \to (S^3, T)$,  with support in a ball neighborhood of $\eta$, as illustrated in Figure \ref{fig:eyeglass1}.  This automorphism is called an {\em eyeglass twist} based on $\eta$.  
\end{defin}

Since the support of an eyeglass twist lies entirely in a ball, it is isotopic in $S^3$ to the identity.  Thus any eyeglass twist of the standard splitting $A \cup_{T_g} B$ lies in $\calG_g$.   

 \begin{figure}[ht!]
  \labellist
\small\hair 2pt
\pinlabel  $\ell_a$ at 50 150
\pinlabel  $\ell_b$ at 235 120
\pinlabel  $v$ at 130 120
\endlabellist
    \centering
    \includegraphics[scale=0.55]{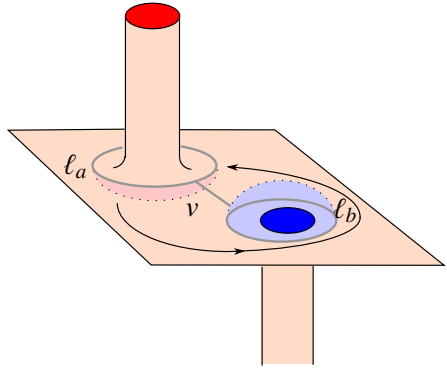}
   \caption{An eyeglass twist} \label{fig:eyeglass1}
    \end{figure}

{\bf Remark:} Figure \ref{fig:eyeglass1} shows the lens disk $\ell_a$ twisting once around the lens disk $\ell_b$, but this is an artifact of the figure; it could equally well have been drawn symmetrically, with $\ell_b$ twisting once around the lens disk $\ell_a$.

\bigskip

Suppose $c \subset T_g$ 
separates $T_g$ into two components, $T_A$ containing the punctured tori in $g_1 < g$ standard bubbles and $T_B$ containing the punctured tori in the other $g_2 = g - g_1$ standard bubbles.  

\begin{lemma}  \label{lemma:eyeglass3o} Suppose $\eta$ is an eyeglass in $T_g$ whose lenses consist of a disk $a \subset A$ with $\bdd a \subset T_A$ and a disk $b \subset B$ with $\bdd b \subset T_B$.  Suppose further that the bridge $v$ intersects $c$ exactly once.  Then an eyeglass twist along $\eta$ is a Powell move. 
\end{lemma} 

\begin{proof} See \cite[Lemma 3.4]{FS1}.
\end{proof}

\begin{cor}  \label{cor:eyeglass3o}  Suppose $\frb$ is a generic bubble disjoint from the standard meridian $\mu_i$ and standard longitude $\lambda_i$ of a standard bubble $\frb_i$, and $\gamma$ is an embedded arc in $T_g - \frb$ from $\bdd \frb$ to $a_i$ (resp. $b_i$) whose interior is disjoint from both $a_i$ and $b_i$.   Then a bubble move  of $\frb$ around the path $\overline{\gamma} a_i\gamma$ (resp $\overline{\gamma} b_i \gamma$) is a Powell move.
\end{cor}

\begin{proof}  Let $c \subset T_g$ be the boundary of a regular neighborhood of $a_i \cup b_i$, separating the standard bubble $\frb_i$ from the rest of $T$.  Then the arc $\gamma$ intersects $c$ in one point and the bubble move around $\overline{\gamma} a_i\gamma$, say, can be described as an eyeglass move along an eyeglass with one lens $\mu_i \subset A$, the other lens the disk $\bdd \frb \cap B$ and the bridge the arc $\gamma$. The result then follows from Lemma \ref{lemma:eyeglass3o}.  
\end{proof}

\section{Stabilization and topological conjugacy}

\begin{defin} \label{defin:generic} Orientation preserving homeomorphisms $h_1, h_2: (S^3, T_g) \to (S^3, T_g)$ are {\em topologically conjugate} if there is an orientation preserving homeomorphism $g:(S^3, T_g) \to (S^3, T_g) \in \calG_g$ so that $g^{-1}h_1g = h_2$.  A homeomorphism topologically conjugate to a standard flip is called a {\em generic flip}; a homeomorphism topologically conjugate to a standard exchange is called a {\em generic exchange}.  
\end{defin}

For example, suppose $\frb$ is {\em any} genus one bubble for $T_g$, so $T \cap \frb$ is a once-punctured unknotted torus whose meridian bounds a properly embedded disk $x \in A$ and longitude bounds a properly embedded disk $y \in B$.  Let $h:(S^3, T_g) \to (S^3, T_g) \in \calG_g$ be the homeomorphism which is the identity outside $\frb$ but looks like Figure \ref{fig:flip} inside $\frb$, with $\mu_1$ replaced by $x$ and $\lambda_1$ by $y$.  Then it follows from \cite{Wa} that $h$ is a generic flip, and from the definition that any generic flip can be described in this way.

Similarly, suppose $\frb$ and $\frb'$ are any two disjoint genus one bubbles for $T_g$, and $\gamma \subset T$ is a path in their complement between the circles $c = \bdd \frb \cap T$ and $c' = \bdd \frb' \cap T$.  A generic exchange is as shown in Figure \ref{fig:switch}, with $\frb_1, \frb_2, v$ replaced respectively with $\frb, \frb', \gamma$.  The generic exchange can also be depicted as a generalization of one of Goeritz's original generators, as in Figure \ref{fig:genswitch}.  In this depiction the visible genus two splitting is rotated $\pi$ around the vertical red axis, as Goeritz did.  But in our case the blue ball at the base is a genus $g-2$ bubble $\frb''$ and the complement of $\frb''$, the genus two splitting we see, is the tube sum of bubbles $\frb$ and $\frb'$ along the arc $\gamma$.  

\begin{figure}[ht!]
  \labellist
\small\hair 2pt
\pinlabel  $\frb''$ at 155 10
\pinlabel  $x$ at 250 80
\pinlabel  $\frb$ at 260 110
\pinlabel  $y$ at 210 70
\pinlabel  $c$ at 165 70
\pinlabel  $c'$ at 120 70
\pinlabel  $\gamma$ at 150 100
\pinlabel  $x'$ at 20 80
\pinlabel  $\frb'$ at 5 110
\pinlabel  $y'$ at 60 70

\endlabellist
    \centering
    \includegraphics[scale=0.7]{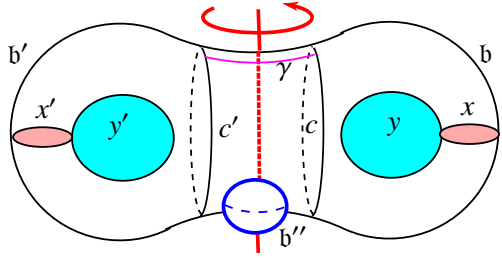}
   \caption{A bubble exchange, via a Goeritz move in $\calG_2$.} \label{fig:genswitch}
    \end{figure}
    
It follows immediately from Lemma \ref{lemma:braid2} that if both $\frb$ and $\frb'$ are standard then the exchange is a Powell move, even though $\gamma$ can be arbitrary.  More surprising and useful is the following extension:

\begin{prop}  \label{prop:semistexchange} If one of the bubbles in a generic exchange is standard, then the exchange is a Powell move.
\end{prop}

We defer the proof in order to describe how the Proposition fits in to the main result of this paper.

\begin{prop} \label{prop:main}  For $\frb_{g+1}$ the standard bubble in $(S^3, T_{g+1})$ suppose the following two assumptions are true:
\begin{enumerate}
\item Any generic bubble exchange between $\frb_{g+1}$ and a disjoint genus one bubble in $(S^3, T_{g+1})$ is a Powell move.  In other words Proposition \ref{prop:semistexchange} is true.

\item Any eyeglass twist in $(S^3, T_{g+1})$ in which the eyeglass frame is disjoint from $\frb_{g+1}$ is a Powell move.  
\end{enumerate} 
Then the function $\iota^+: \calG_g \to 
    \calG_{g+1}/\calP_{g+1}$ defined in Proposition \ref{prop:Ginclude} is trivial.
\end{prop}

\begin{proof}  We begin with two claims, assuming the two                             assumptions are true:
\medskip

{\em Claim 1:}  Suppose $\frb$ is a bubble in $T_{g+1}$ that is disjoint from $\frb_{g+1}$.  Then the generic flip in $\frb$ is a Powell move.  

{\em Proof of claim:}  Recall from the discussion before Proposition \ref{prop:newgen} that  the standard flip $\omega_{g+1}$ of $\frb_{g+1}$ is a Powell move.  Suppose $\frb$ is a generic bubble disjoint from $\frb_{g+1}$.  Choose an embedded path $\beta$ from $\frb_{g+1}$ to $\frb$ and let $\rho$ be the bubble exchange of $\frb_{g+1}$ and $\frb$ along $\beta$, a Powell move by assumption (1).  Then $\rho^{-1} \omega_{g+1} \rho$ is the generic flip of $\frb$ and is the composition of Powell moves. This proves the claim.
\medskip 

{\em Claim 2:}  Suppose $\frb_1, \frb_2$ are disjoint bubbles in $T_{g+1}$ that are also disjoint from $\frb_{g+1}$.  Suppose $\alpha \subset T_2$ is an arc between them that is also disjoint from $\frb_{g+1}$.  Then the generic exchange of $\frb_1$ with $\frb_2$ along $\alpha$
 is a Powell move.

{\em Proof of claim:}  
Consider the 3-punctured genus $g-2$ surface $T_- = T_{g+1} - {\frb_1 \cup \frb_2 \cup \frb_{g+1}}$.  For $i = 1, 2$ let $\beta_i \subset T_-$ be an embedded arc from $\frb_i$ to $\frb_{g+1}$, chosen so that $\alpha, \beta_1, \beta_2$ are all disjoint and lie in a disk in $T_- $.  

By assumption (1) each bubble exchange $\rho_i$ along $\beta_i$ is a Powell move.  Thus the composition $\rho_1^{-1} \rho_2 \rho_1$ is a Powell move. But  $\rho_1^{-1} \rho_2 \rho_1:T_- \to T_-$ is easily seen to be isotopic to the exchange of the $\frb_i$ along $\alpha$, proving the claim.  (The composition $\rho_2 \rho_1 \rho_2^{-1}$ is also isotopic to this exchange. 
Indeed the consequent isotopy between  $\rho_1^{-1} \rho_2 \rho_1$ and $\rho_2 \rho_1 \rho_2^{-1}$ acting on a 3-punctured disk in $T_-$ is the source of the standard braid relation $\sigma_i \sigma_{i+1} \sigma_i = \sigma_{i+1} \sigma_i\sigma_{i+1}$.)
\medskip

Following the claims we exploit the central result of \cite{Sc2}:  the Goeritz group $\calG_g$ is generated by eyeglass twists and topological conjugates of generators of $\calP_g$. We consider the generators of $\calP_g$ from Proposition \ref{prop:newgen}.  These consist of Powell's $D_{\theta}$, which is itself an eyeglass twist, Powell's $D_{\omega}$ which is a standard flip, and the collection $\{\phi_i\}$ of standard exchanges.  Any topological conjugate of an eyeglass twist is an eyeglass twist, any topological conjugate of a standard flip is a generic flip and any topological conjugate of a standard exchange $\phi_i$ is a generic exchange.  Assumption (2) above says that $\iota^+$ takes any eyeglass twist in $\calG_g$, (its frame disjoint from the point $\star$ by general position) to $\calP_{g+1} \subset \calG_{g+1}$.  Similarly, Claim 1, using Assumption (1), says that  $\iota^+$ takes any generic flip in $\calG_g$ to $\calP_{g+1}$ and Claim 2, also using Assumption (1), says that $\iota^+$ takes any generic exchange in $\calG_g$ to $\calP_{g+1}$.  Since each generator of  $\calG_g$ is taken to $\calP_{g+1}$, $\iota^+: \calG_g \to \calG_{g+1}/\calP_{g+1}$ is trivial, as required.
\end{proof}

We now show that in fact both assumptions in Proposition \ref{prop:main} are true.
The proofs of both are highly visual.  We first prove Assumption (2) of Proposition \ref{prop:main}.

\begin{prop}  \label{prop:ass2} Suppose $\tau \in \calG_{g+1}$ is an eyeglass twist whose frame $\eta$ is disjoint from the standard bubble $\frb_{g+1}$.  Then $\tau$ is a Powell move.
\end{prop}

\begin{proof} Let $\{\ell_a, \ell_b, v\}$ be the frame of $\eta$.  Let $u$ be an arc such that
\begin{itemize}
\item the ends of $u$ lie on $a_{g+1}$ near the point $a_{g+1} \cap b_{g+1}$ and on $\ell_a$ near the point $\ell_a \cap v$,
\item $u$ crosses the circle $c=\bdd \frb_{g+1}$ once and
\item $u$ is otherwise disjoint from $a_{g+1} \cup b_{g+1}$, and $\eta$.
\end{itemize}
  See Figure \ref{fig:stabilized}.
 
 \begin{figure}[ht!]
\labellist
\small\hair 2pt
\pinlabel  $\ell_a$ at 110 145
\pinlabel  $\ell_b$ at 240 100
\pinlabel  $v$ at 185 145
\pinlabel  $u$ at 265 170
\pinlabel  $a_{g+1}$ at 300 190
\pinlabel  $b_{g+1}$ at 400 230
\pinlabel  $c$ at 365 150
\endlabellist
    \centering
    \includegraphics[scale=0.45]{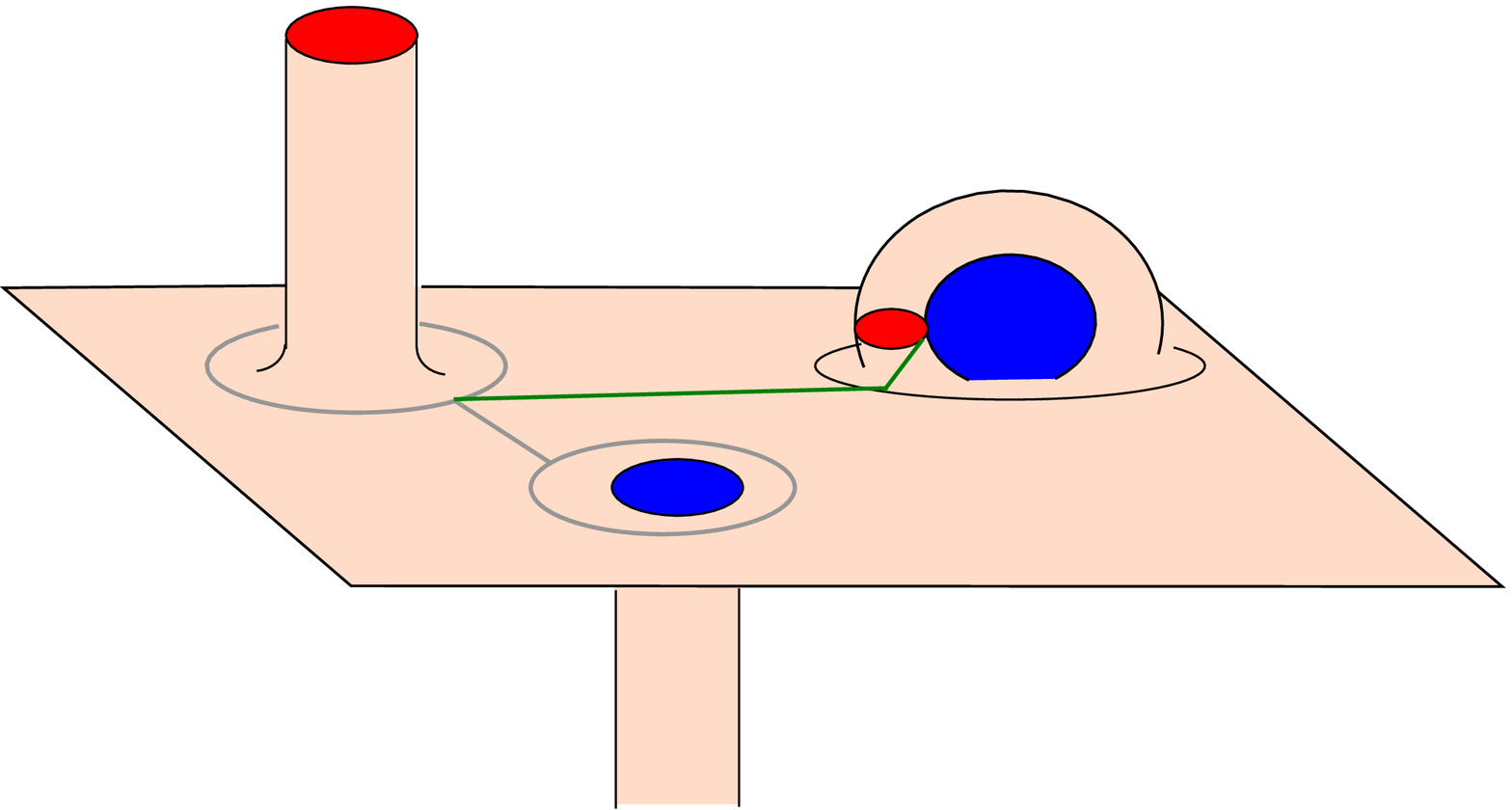}
     \caption{} \label{fig:stabilized}
    \end{figure}
    
        \begin{figure}[ht!]
\labellist
\small\hair 2pt
\pinlabel  $\eta_+$ at 190 125
\pinlabel  $\eta_1$ at 265 175
\pinlabel  $\bdd b_{g+1}$ at 450 185
\endlabellist
    \centering
    \includegraphics[scale=0.45]{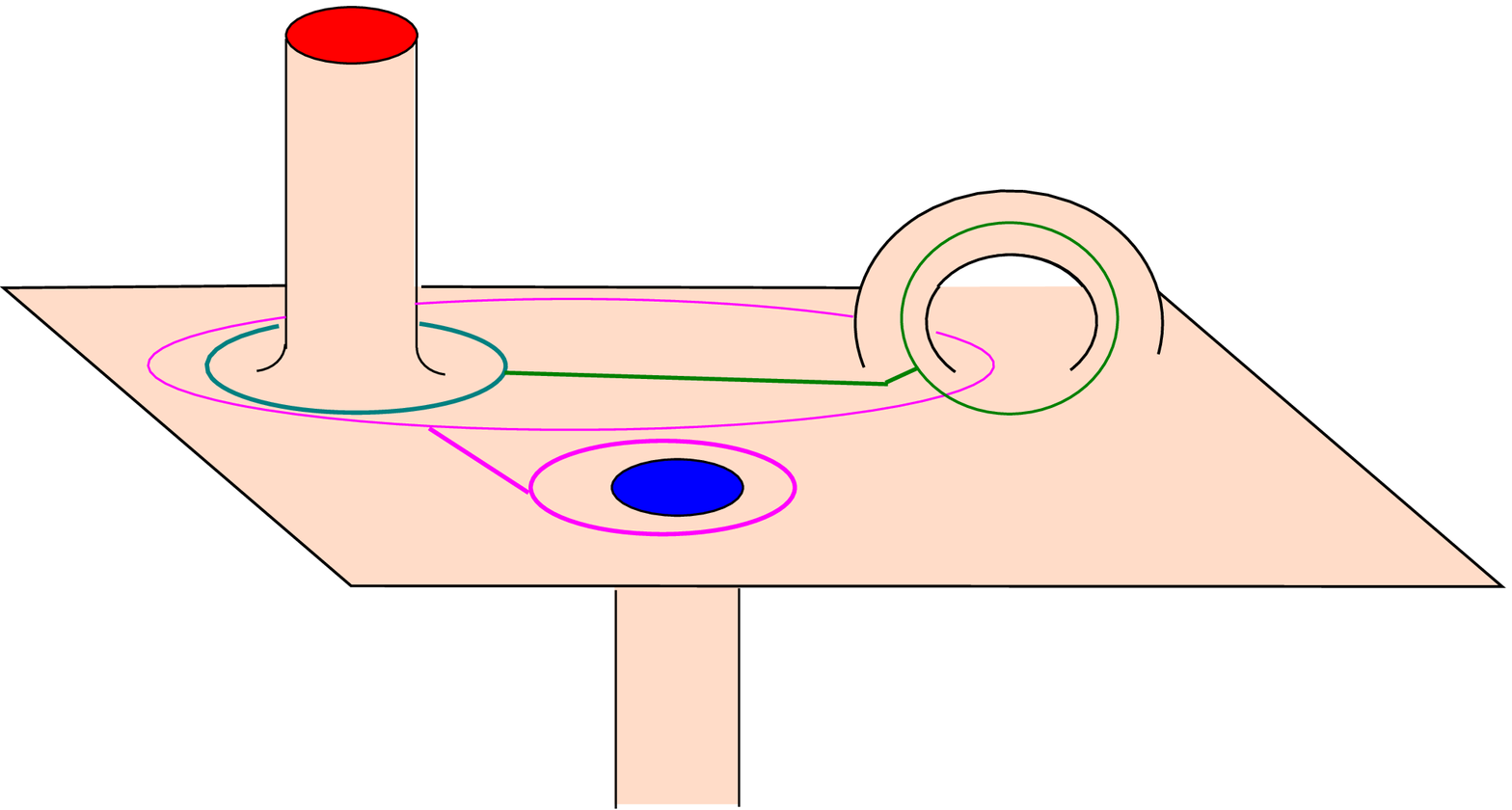}
     \caption{} \label{fig:stabilized2}
    \end{figure}

\begin{figure}[ht!]
\labellist
\small\hair 2pt
\pinlabel  $\tau^{-1}_1(\eta_+)$ at 150 470
\pinlabel  $\eta'_+$ at 265 180
\endlabellist
    \centering
    \includegraphics[scale=0.45]{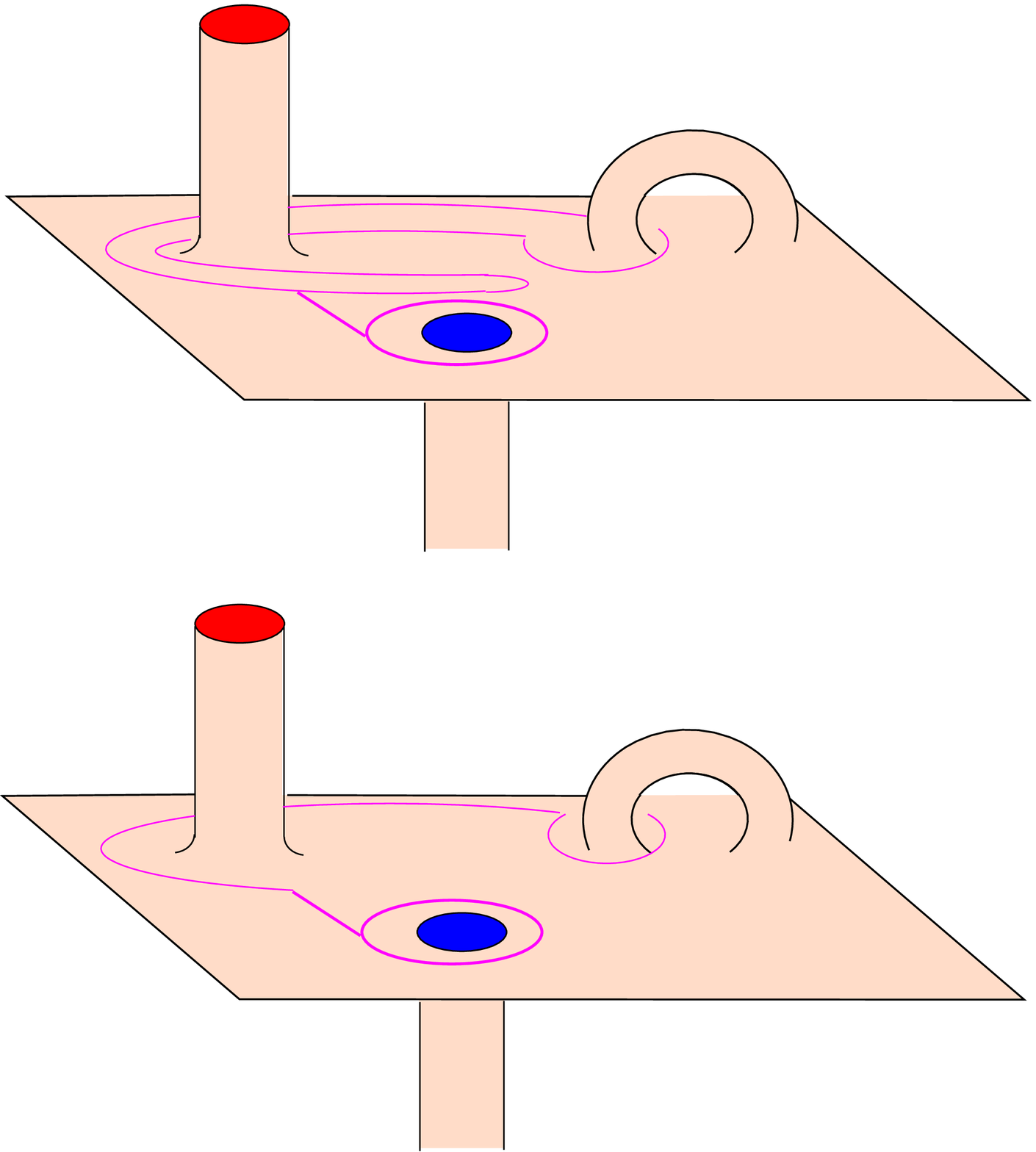}
     \caption{} \label{fig:stabilized3}
    \end{figure}
   
    Let $\eta'$ be the eyeglass twist whose eyeglass is $\{\mu_{g+1}, \ell_b, u \cup v\}$.  We know from Lemma \ref{lemma:eyeglass3o} that an eyeglass twist $\tau'$ along $\eta'$ {\bf is a Powell move}.  Let $\ell'_a$ be the band sum of $\ell_a$ and $\mu_{g+1}$ along $u$ and observe, by watching the motion of $\ell_b$, that the composition $\tau \tau'$ is an eyeglass twist $\tau_+$ whose eyeglass is $\eta_+$ is $\{\ell'_a, \ell_b, v\}$.

Now let $\eta_1$ be the eyeglass given by $\{\ell_a, \lambda_{g+1}, u\}$ and let $\tau_1$ be the eyeglass twist along $\eta_1$.  See Figure \ref{fig:stabilized2}.  Again, from Lemma \ref{lemma:eyeglass3o}, $\tau_1$ {\bf is a Powell move}.  Further observe that $\tau^{-1}_1(\eta_+)$ is an eyeglass $\eta'_+$ with lenses $\mu_{g+1}$ and $\ell_b$ and a bridge that intersects $c$ in a single point, so again an eyeglass twist $\tau'_+$ along $\eta'_+$ {\bf is a Powell move}.  See Figure \ref{fig:stabilized3}.  

Since $\tau^{-1}_1(\eta_+) = \eta'_+$ it follows that $\tau_+ = \tau_1\tau'_+ \tau^{-1}_1$. 
As a composition of Powell moves, $\tau_+$ {\bf is a Powell move}.  But we have earlier shown that $\tau = \tau_+ \tau'^{-1}$ so, as a composition of Powell moves, $\tau$ is a Powell move.  \end{proof}
    
    \begin{figure}[ht!]
\labellist
\small\hair 2pt
\pinlabel  $\eta_1$ at 200 680
\pinlabel  $\eta_2$ at 200 655
\pinlabel  $a_{g+1}$ at 260 720
\pinlabel  $b_{g+1}$ at 260 680
\pinlabel  $x$ at 130 720
\pinlabel  $y$ at 130 680
\pinlabel  $\tau_1$ at 205 600
\pinlabel  $\tau_2$ at 210 480
\pinlabel   {\text flip }$\omega_{g+1}$ at 230 370
\pinlabel   $\tau_1^{-1}$ at 220 240
\pinlabel   {\text standard bubble\;   move of $\frb_{g+1}$ around $x$} at 220 110
\endlabellist
    \centering
    \includegraphics[scale=0.75]{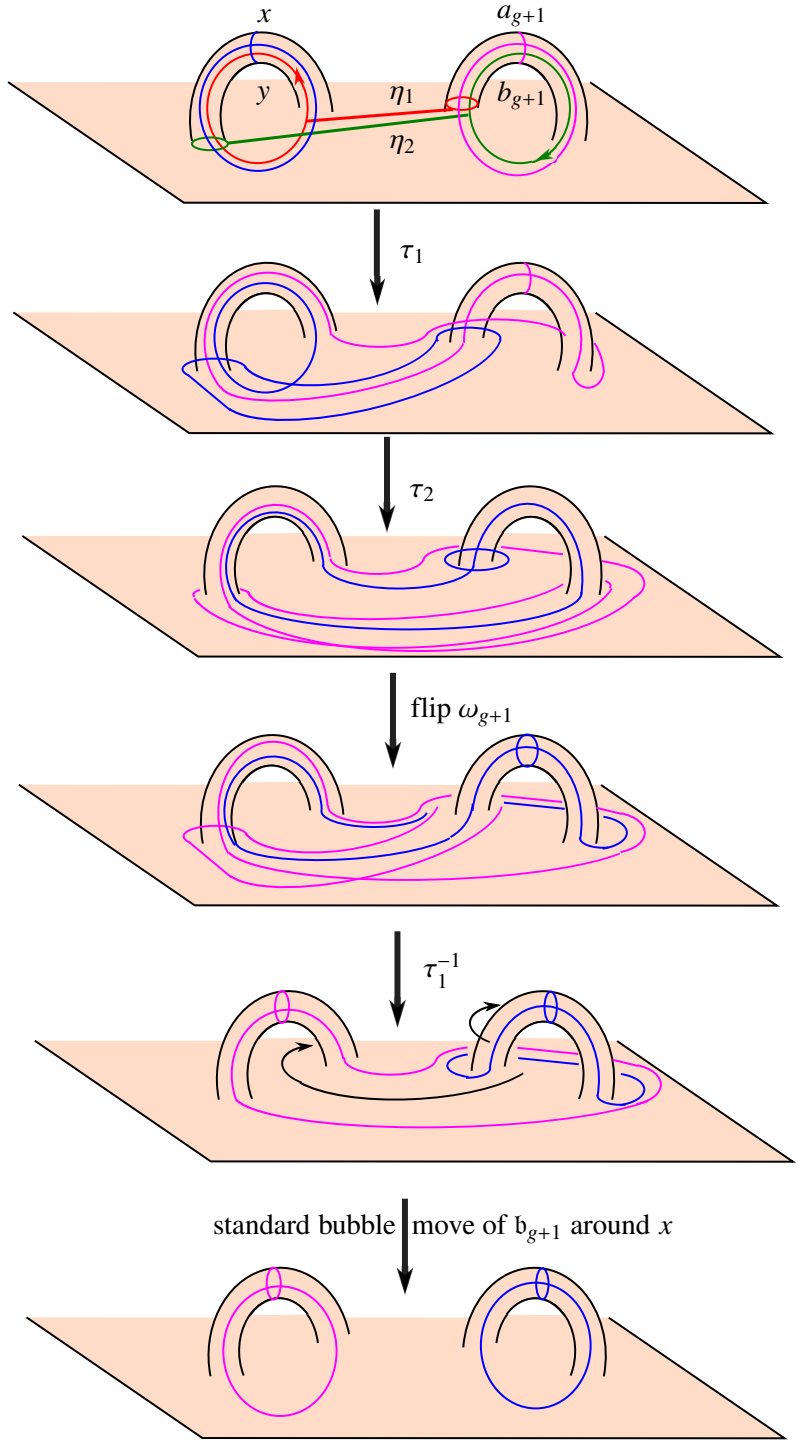}
     \caption{An exchange that is a Powell move} \label{fig:semiexch}
    \end{figure}
    
    \begin{prop}  \label{prop:ass1} Suppose $\tau \in \calG_{g+1}$ is a bubble exchange between the standard bubble $\frb_{g+1}$ and a generic bubble $\frb$.    Then $\tau$ is a Powell move.
\end{prop}

\begin{proof}  The entire proof is contained in Figure \ref{fig:semiexch}.  In the top panel the standard bubble $\frb_{g+1}$ is on the right, with the meridian circle $a_{g+1}$ and the longitude circle $b_{g+1}$ both shown in fuchsia.  On the left is a generic bubble with the boundaries of meridian $x \subset A$ and longitude $y \subset B$ shown in blue.  Also shown there are two eyeglasses, $\eta_1$ and $\eta_2$ in red and green respectively.  Since the bubble on the right is standard, by Lemma \ref{lemma:eyeglass3o} the eyeglass twists $\tau_i$ along $\eta_i$ are both Powell moves.  The direction of the twist along the $\eta_i$ that we will use is shown by arrows on the longitudinal lenses of $\eta_1$ and $\eta_2$, parallel respectively to $\bdd y$ and $b_{g+1}$.  

The next two panels down show the effect of first $\tau_1$ then $\tau_2$.  The following panel shows the effect of a (clockwise $\pi$) flip on the standard bubble $\frb_{g+1}$, a Powell move.  Then follow two Powell moves: $\tau_1^{-1}$ and a final move of the standard bubble around the meridian $x$ of the generic bubble.  The final result in the last panel is that the pair $\mu_{g+1}, \lambda_{g+1}$ and the pair $x, y$ have been exchanged.  

Note that the orientations of $x, y$ and $\mu_{g+1}, \lambda_{g+1}$ are not an issue: if an exchange as above moves an oriented $\mu_{g+1}$ to an unwanted orientation of $x$ (or vice versa) then precomposing (or postcomposing) with a standard flip (a Powell move) will fix the problem.
\end{proof}

\begin{cor} The function $\iota^+: \calG_g \to 
    \calG_{g+1}/\calP_{g+1}$ defined in Proposition \ref{prop:Ginclude} is trivial.\qed
\end{cor}

{\bf Two final notes:}

The proof of Proposition \ref{prop:ass1} focuses on the standard bubble $\frb_{g+1}$ and the generic bubble $\frb$, both lying in the neighborhood of a disk $D$ that contains the exchange; the exchange map is the identity map outside of the neighborhood of $D$.  Left undiscussed is the pair of pants $D - (\frb \cup \frb_{g+1}) \subset T_{g+1}$. Observe, though, that any homeomorphism of this pair of pants, fixed on its boundary, is also a Powell move.  Here is the argument:  Any such homeomorphism can be constructed by Dehn twisting some number of times around each of the boundary components of the pair of pants.  A Dehn twist at the boundary of the standard bubble $\frb_{g+1}$ is a Powell move, for it is a composition of two standard flips.  Similarly, conjugating this Powell Dehn twist by the exchange just described, also a Powell move, gives a a Dehn twist around the bubble $\frb$, so such a Dehn twist is also Powell.  Lastly, its easy to check that the exchange defined above gives a simple (clockwise) half-twist on a collar of $\bdd D$, so a full Dehn twist around that boundary component can be accomplished just by doing the exchange twice.  
\medskip

Because the exchange in Proposition \ref{prop:ass1} moves $\frb_{g+1}$ to another bubble, the exchange itself is not in the image of $\iota^+$ and its construction may seem a bit {\em ad hoc}. The following appendix puts the bubble exchange into a broader context, one involving an order 12 dihedral subgroup of $\calG_2$, each of whose elements somewhat naturally defines an element of $\calG_{g+1}$ that also lies in $\calP_{g+1}$ but not in the image of $\iota^+$.   Figure \ref{fig:genswitch} shows the connection between one such element of $\calG_2$ and a bubble exchange like that of Proposition \ref{prop:ass1}.
\newpage

 \setcounter{section}{0} 

\begin{center}
\large 
\textsc{Appendix: Proposition \ref{prop:ass1} in a larger context}
\end{center} 

\renewcommand{\thesection}{A\arabic{section}}

\section{Symmetries of $K_{2, 3}$ as a subgroup of $\calG_2$} \label{sect:dih}

 \begin{figure}[ht!]
  \labellist
\small\hair 2pt
\pinlabel  $z$ at 230 280
\pinlabel  $x$ at 320 80
\pinlabel  $y$ at 300 230
\pinlabel  $q_0$ at 275 215
\pinlabel  $q_2$ at 235 55
\pinlabel  $q_1$ at 20 155
\pinlabel  $\frc_0$ at 180 120
\pinlabel  $\frc_2$ at 245 135
\pinlabel  $\frc_1$ at 80 245
\pinlabel  $\frc_e$ at 100 220
\endlabellist
    \centering
    \includegraphics[scale=0.6]{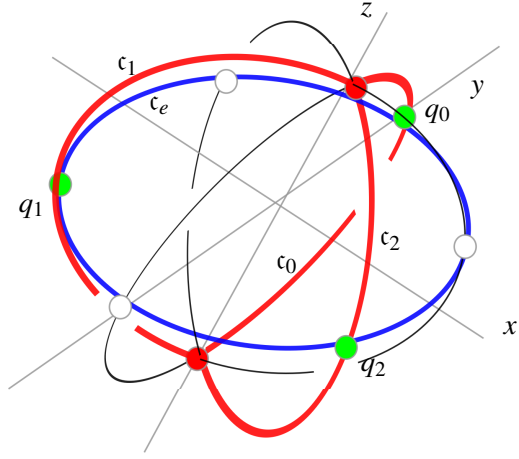}
   \caption{A highly symmetric placement of $T_2$ in $S^3$} \label{fig:involS3x}
    \end{figure}

Consider the complex shown in Figure \ref{fig:involS3x}.  It is embedded in the unit sphere $S^2 \subset S^3$ viewed as the boundary of a ball in $\real^3 \subset S^3$, with axes in $\real^3$ as shown.  There are four great circles on $S^2$, three of them vertical circles $\{ \frc_0, \frc_1, \frc_2\}$ that pass through the poles, and, respectively, through the (green) points $q_0, q_1, q_2$ in the  equator $\frc_e$.  The points $q_0, q_1, q_2 \in \frc_e$ are respectively at an angle of $\{ 0, 2\pi/3, 4\pi/3 \}$ from the point $(0, 1, 0)$ on the $y$ axis.  Thus each $q_i, i \in \zed_3$ is one of the two points in $\frc_i \cap \frc_e$.  

Let $K$ denote the subcomplex consisting of the two poles and the three points $q_i$, together with the semicircles in each of the $\frc_i, i \in \zed_3$ that contain the 
$q_i$. $K$ can be thought of as the complete bipartite graph $K(2, 3)$.  A regular neighborhood of $K$ (shown in red) is a genus 2 handlebody we will denote $A$.  Its complement in $S^3$ is a genus 2 handlebody we denote $B$.  Thus $K$ determines a genus 2 Heegaard splitting of $S^3$.  

Consider the group $\calG_K \subset \calG_2$ generated by $\pi$-rotations of $S^3$ around these four great circles, as these rotations act on $A$.  We denote these $\pi$-rotations by $\rho_e, \rho_0, \rho_1, \rho_2$ where $\rho_e$ is $\pi$-rotation around $\frc_e$ and each $  \rho_i, i = 0, 1, 2$ is $\pi$-rotation around $\frc_i$.   Here are some obvious properties:

\begin{enumerate}
\item Each $\pi$-rotation interchanges the interior of the unit ball with the exterior.
\item $\rho_e$ is the hyperelliptic involution on $A$.  It leaves the $q_i$ fixed but interchanges the poles.
\item Each $\rho_i$ leaves $\frc_i$ and $\frc_e$ setwise fixed but transposes the other two circles $\frc_j, j \neq i$.  Thus it leaves the poles fixed, but transposes two of the $q_i$.
\end{enumerate}

Daryl Cooper \cite{Co} has shown me a pleasant orbifold picture of $S^3/\calG_K$, see \cite{CHK} and Figure \ref{fig:orbifold}.  The 3 red half-circles and the red dots again represent $K$, whose regular neighborhood has boundary $T_2$.  The lines through the blue origin meet at a blue point at infinity not shown.  The stabilizer of the $z$-axis is the dihedral group $\Dih_3$ of order 6.  The stabilizers of each of the green dots $q_0, q_1, q_2$ (and of their antipodal white dots) are Klein 4-groups.  

The unit sphere divides $S^3$ into two balls, one inside and one outside, and is itself divided by the circles $\{\frc_e, \frc_0, \frc_1, \frc_2\}$ into 12 congruent triangles with angles $\pi/3, \pi/2, \pi/2$.
The unit ball is divided into 12 spherical 3-simplices by taking the cone from the blue origin to
these spherical triangles.
The ball that is the complement of the unit ball is also divided into 12 spherical triangles using the cone from blue $\infty$.  $\calG_K$ preserves this decomposition of $S^3$ into 24 spherical 3-simplices.
A fundamental domain $D$ for $\calG_K$ is the suspension from the blue dots (one at $\infty$) of the spherical triangle with vertices green, white, red.  Thus $D$ is two spherical 3-simplices identified along that triangle.   The action of $\calG_K$ identifies points on $\bdd D$
to give a space that is geometrically the double of one of the 3-simplices.

\begin{figure}[ht!]
  \labellist
\small\hair 2pt
\pinlabel  $2$ at 480 180
\pinlabel  $2$ at 520 160
\pinlabel  $2$ at 530 250
\pinlabel  $2$ at 480 120
\pinlabel  $2$ at 600 160
\pinlabel  $3$ at 450 220
\pinlabel  $y$ at 380 195
\pinlabel  $z$ at 210 340
\pinlabel  $\frc_e$ at 110 130
\pinlabel  $\frc_0$ at 300 300
\pinlabel  $\frc_1$ at 220 280
\pinlabel  $\frc_2$ at 180 280
\pinlabel  $q_2$ at 170 140
\pinlabel  $q_1$ at 140 250
\pinlabel  $q_0$ at 320 170
\endlabellist
    \centering
    \includegraphics[scale=0.5]{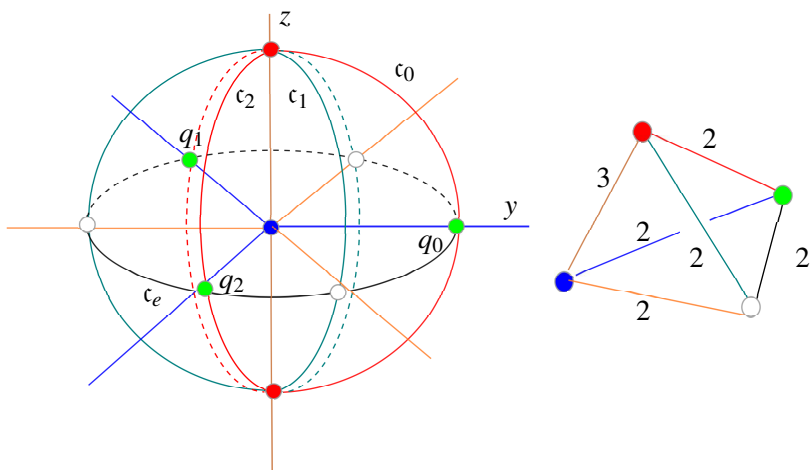}
   \caption{Orbifold picture of $S^3/\calG_K$} \label{fig:orbifold}
    \end{figure}

We will show that  $\calG_K  \subset \calG_2$ is isomorphic to the symmetry group of 
$K_{2, 3}$, namely the product of $\zed_2$ (the symmetry group of two elements, in our case the poles) and $\Dih_3$ the six element dihedral group that is the symmetry group of three elements, in our case the $q_i$.  Thus the symmetry group of $K(2, 3)$ is the 12 element dihedral group $\Dih_6$.  

%

\begin{lemma}  $\calG_K \cong \Dih_6$.
\end{lemma}

\begin{proof}  We have noted above that there is a natural homomorphism $\calG_K \to \Dih_6$, namely $\rho_e$ transposes the poles and each $\rho_i$ transposes the two points $q_j, j \neq i$. The homomorphism is an epimorphism because transpositions of the $q_i$ generate $\Dih_3 \subset Dih_6$.  Finally, suppose that a composition  $g \in \calG_K$ of the $\pi$-rotations is the identity on $K(2, 3)$.  In particular, it is the identity on both the equator of $S^2$ and the poles.  Each $\pi$-rotation is a transposition in $K(2, 3)$ and also interchanges the interior of the ball with the exterior.  If the result is the identity on $K(2, 3)$ then there are an even number of transpositions, hence $g$ does not exchange the interior of the ball with the exterior, and so is the identity on all of $S^3$.
\end{proof}

Following the Lemma, we will henceforth denote $\calG_K$ by $\Dih_6$.

 \begin{figure}[ht!]
  \labellist
\small\hair 2pt
\pinlabel  $z$ at 70 160
\pinlabel  $\rho_z=\rho_0$ at 250 180
\pinlabel  $\rho_x=\rho_e$ at 360 45
\pinlabel  $\rho_{\theta}=\rho_0\rho_2$ at 430 170
\pinlabel  $q_2$ at 135 90
\pinlabel  $q_1$ at 0 90
\pinlabel  $q_0$ at 70 90
\pinlabel  $x$ at 155 90
\pinlabel  $a_2$ at 475 85
\pinlabel  $b_2$ at 408 85
\endlabellist
    \centering
    \includegraphics[scale=0.8]{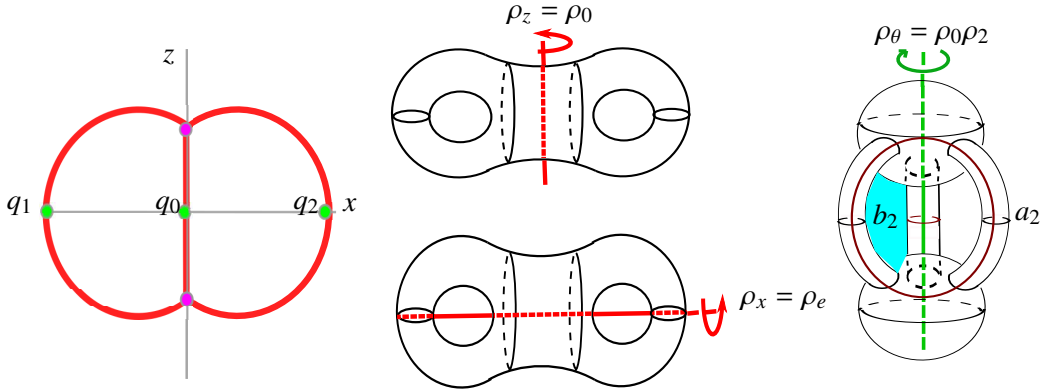}
   \caption{$\Dih_6$ acting on the genus 2 handlebody (planar viewpoint) } \label{fig:k23}
    \end{figure}

There is another point of view on $K$, one that is useful for understanding the stabilizer of a vertex $q_i$.  Such a stabilizer is a subgroup of $\Dih_6$ that is isomorphic to the Klein 4-group.  Figure \ref{fig:k23} examines, as an example, the subgroup that stabilizes $q_0$.  Here $K$ is placed in a plane with $q_0$ at the origin; we will call it the {\em planar viewpoint}.  It is obtained in the first panel of Figure \ref{fig:k23} by deforming $S^2 \subset S^3$ so that it lies on the $xz$-plane.  Then $\frc_e$ becomes the $x$-axis, $\frc_0$ becomes the $z$-axis and the $y$-axis points into the page.  From this point of view, as shown in the second panel, $\rho_0$ becomes $\pi$-rotation $\rho_z$ around the $z$-axis and the hyperelliptic involution $\rho_e$ becomes $\pi$-rotation around the $x$-axis $\rho_x$.  So then $\pi$-rotation around the $y$-axis is $\rho_y = \rho_x\rho_z = \rho_e\rho_0$.  (Whether $\rho_x, \rho_y, \rho_z$ are represented by a clockwise or counterclockwise rotation about the axis is immaterial: the difference represents the non-trivial element of $\pi_1(S0_3)$, and so is non-trivial in $\tilde{\calG}_2$ but trivial in $\calG_2$.)  The third panel returns to the original positioning of the $q_i$, and notes, as an example, that a $-2\pi/3$ rotation $\rho_{\theta}$ corresponds to the composition $\rho_0\rho_2$, each determining a transposition in $\{q_i\}$.  

There is a natural notation for elements of the symmetry group of $K(2, 3)\cong\Dih_6$, namely if an element of the symmetry group transposes the two-element part, associate the sign $-$; if not associate the sign $+$.  Then use  the standard cycle notation for the action on the three-element part of $K(2, 3)$.  Thus in Figure \ref{fig:k23} (and $\rho_y$ rotation around the $y$-axis pointing into the page) we have 

\[\rho_{\theta} = +(021) \quad \rho_z = +(12) \quad \rho_x = -(\; ) \quad \rho_y = -(12)\]

 \begin{figure}[ht!]
  \labellist
\small\hair 2pt
\pinlabel  $a_0$ at 95 60
\pinlabel  $a_2$ at 170 60
\pinlabel  $a_1$ at 20 60
\pinlabel  $b_0$ at 100 110
\pinlabel  $b_1$ at 140 50
\pinlabel  $b_2$ at 50 50
\endlabellist
    \centering
    \includegraphics[scale=1]{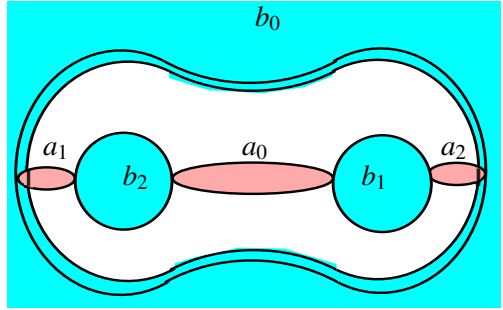}
   \caption{Labelled meridians of $A$ and $B$} \label{fig:Kdef2}
    \end{figure}
    
    As described above, each of the genus 2 handlebodies $A$ and $B$ has a natural set of 3 meridian disks, called {\em labeled meridians}.  Those for $A$ are labelled $a_0, a_1, a_2$, each $a_i, i \in \zed_3$ located where the $xy$-plane intersects $A$ at $q_i$.  Similarly the meridians for $B$ are labeled $b_0, b_1, b_2$, each $b_i$ consisting of the disk component of $S^2 - A$ that is centered on the antipode of $q_i$ in $S^2$.  The meridian disks $a_2 \subset A$ and $b_2 \subset B$ are named in the third panel of Figure \ref{fig:k23}; the boundary of the meridian $b_0$ is also drawn there.  
    
    Figure \ref{fig:Kdef2} shows the labeled meridians as they appear in the planar viewpoint.  From this perspective the meridians $\{a_i\}, i \in \zed_3$ are $A \cap (xy{\rm -plane})$ and  the meridians $\{b_i\}$ are $B \cap (xz{\rm -plane} \cup \infty)$.  
    
We have these elementary observations:

\begin{lemma} \label{lemma:Kdef}  Let $\kappa \subset T_2$ be the $1$-complex that is the union of the 6 boundary circles of the labeled meridians.  
\begin{enumerate}
\item Each of the sets  $\{a_i\}$ and  $\{b_i\}$, $ i \in \zed_3$, of labeled meridians are  setwise invariant under the action of $\Dih_6$ and therefore $\kappa$ is also. 
\item A labeled $A$-meridian is disjoint from exactly one labeled $B$-meridian and meets each of the other two labeled $B$-meridians in a single point (and symmetrically).  
\item A labeled $A$-meridian and a labeled $B$-meridian that meet in a single point are called an {\em orthogonal pair of labeled meridians}; each such pair defines a genus 1 bubble in $T_2$.  
\item The complement in the surface $T_2$ of any two orthogonal pairs of labeled meridians is connected. Two examples: $T_2 - [(a_2 \cup b_1) \cup (a_1 \cup b_2)]$ is an annulus and $T_2 - [(a_2 \cup b_1) \cup (a_0 \cup b_2)]$ is a disk. \qed
\end{enumerate} 
\end{lemma}  

Regarding (4), we do not require that the two orthogonal pairs be disjoint, so their union may intersect $T$ in either two, three or four circles, each lying in $\kappa$.

\section{From $\calG_2$ to $\calG_g$ - first observations} \label{sect:firstobs}

The operation of $\Dih_6$ on $(S^3, T_2)$ just described  
 will be useful in understanding bubble exchanges in $\calG_g$, as we will explain in the next few sections. 

Pick a point $p_{\frc}$  on the curve $\frc$ in the standard genus $g-2$ Heegaard surface $T_{g-2}$.  The complement of a small $3$-ball neighborhood of $p_{\frc}$ is a genus $g-2$ bubble. Choose a homeomorphism of this bubble to a small ball in $\Rrr^3$ centered at $0$, and denote the ball $\frb_{\frc}$.  For any point $p \in T_2$ consider the genus $g$ Heegaard surface of $S^3$ obtained by replacing a small ball neighborhood in $S^3$ of $p \in T_2$ by the bubble $\frb_{\frc}$, translated in $\Rrr_3$ to $p$. See Figure \ref{fig:tgp}.  

The sphere $\bdd \frb_{\frc}$, when displaced to $p$, typically intersects each of $T_2$ and the original $T_{g-2}$ in different circles on $\bdd \frb_{\frc}$, but they can be unambiguously deformed to match up:  Normally orient the circle $\bdd \frb_c \cap T_2$ from the handlebody $A$ to the handlebody $B$ and do the same for the circle $\bdd \frb_c \cap T_{g-1}$.  Then isotope one circle to the other in a collar of $\bdd \frb_{\frc} \subset T_2$.  This can be done unambiguously, because the space of oriented great circles in the 2-sphere is just $S^2$ and so is simply connected.  Denote the resulting genus $g$ Heegaard surface by $T(p)$.  

 \begin{figure}[ht!]
  \labellist
\small\hair 2pt
\pinlabel  $\frb_{\frc}$ at 140 85
\pinlabel  $T_2$ at 140 5
\pinlabel  $p$ at 275 80
\pinlabel  $T_g(p)$ at 430 5
\endlabellist
    \centering
    \includegraphics[scale=0.7]{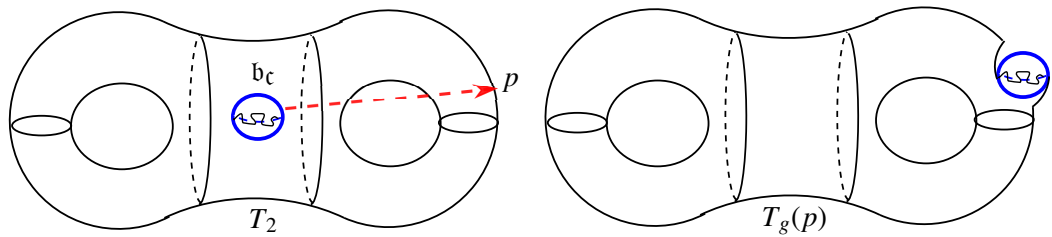}
   \caption{Construction of Heegaard surface $T_g(p)$} \label{fig:tgp}
    \end{figure}
    
     Let $\gamma \subset T_2$ be an embedded arc with end points $p, q \in T_2$.  An easy and classical argument (echoing that just before Proposition \ref{prop:Ginclude}) gives a homeomorphism $(S^3, T_2, p) \to (S^3, T_2, q)$ with support in a ball neighborhood of $\gamma$, basically a push within that ball, see \cite[Homogeneity Lemma]{Mi}.  Replacing ball neighborhoods of $p$ and $q$ respectively by copies of $\frb_{\frc}$ the homeomorphism gives a natural homeomorphism of pairs $\gamma^q_p: (S^3, T(p)) \to (S^3, T(q))$.  Here are some informal observations, which we collect as a multipart Lemma:
\begin{lemma} \label{lemma:gammapq} For $\gamma$ as above:
\begin{enumerate} 
\item If $\gamma' \subset T_2$ is an embedded arc isotopic rel end points in $T_2$ to $\gamma$ then 
    ${\gamma'}^q_p: (S^3, T(p)) \to (S^3, T(q))$ is isotopic as a map of pairs to $\gamma^q_p$
, a relation we denote ${\gamma'}^q_p\sim \gamma^q_p$.  
\item $ \gamma^p_q \gamma^q_p \sim identity: (S^3, T(p)) \to (S^3, T(p))$.
\item If $\gamma' \subset T_2$ is an embedded arc in $T$ with ends at points $q$ and $r$ and with interior disjoint from $\gamma$ then ${\gamma'}^r_q \gamma^q_p\sim (\gamma \cup \gamma')^r_p: (S^3, T(p)) \to (S^3, T_r)$.  

\item Suppose, as just described, $\gamma' \subset T_2$ is an embedded arc in $T$ with ends at points $q$ and $r \neq p$, but with the interior of $\gamma'$ {\em not necessarily} disjoint from $\gamma$.  Observe that then $\gamma$ can be isotoped in $T_2$, rel end points, to be disjoint from $\gamma'$ by piping points of $\gamma \cap \gamma'$ to, and then beyond, the end of $\gamma'$ at $r$.  This has no effect (up to isotopy of pairs) on $\gamma^q_p$ so, following (3), we may as well {\em define} $(\gamma \cup \gamma')^r_p: (S^3, T(p)) \to (S^3, T_r)$ as $ {\gamma'}^r_q \gamma^q_p$.  Note that the arc $  \gamma \cup \gamma'$ is no longer necessarily embedded.

\item Generalizing (4), suppose $\gamma \subset T_2$ is a {\em not necessarily embedded} arc with end points $p, q \in T_2$, in general position. (General position will typically be assumed in arguments below.) Then $\gamma^q_p: (S^3, T(p)) \to (S^3, T(q))$ can be well-defined (up to isotopy of pairs) by subdividing $\gamma$ into a finite collection of arcs, each of which is embedded, and setting $\gamma^q_p$ to be the corresponding composition of homeomorphisms given by the embedded arcs.

\item Suppose $\gamma'$ is another (general position) arc with end points $p$ and $q$ and $\gamma'$ is merely {\em homotopic} to $\gamma$ rel end points.  Then, after subdividing $\gamma'$ sufficiently, as in (5), the homotopy can be realized as a sequence of isotopies of subarcs of $\gamma'$.  It follows that $\gamma'^q_p \sim \gamma^q_p$.
\item if $\gamma$ is a loop in $T$ based at $p$ then $\gamma^p_p$ is a generic bubble move around the loop.\qed
\end{enumerate}
\end{lemma}

To recapitulate and emphasize: following Lemma \ref{lemma:gammapq}(5), $\gamma^q_p: (S^3, T(p)) \to (S^3, T(q))$ can be defined even when the arc $\gamma$ is not embedded and, per Lemma \ref{lemma:gammapq}(6), $\gamma^q_p$ is determined (up to isotopy of pairs) by the homotopy class of $\gamma$ rel its endpoints.

\begin{lemma} \label{lemma:homeocommute}  Suppose $p, q \in T_2$, $\psi: (S^3, T_2) \to (S^3, T_2)$ is a homeomorphism, with \[\psi_p: (S^3, T(p)) \to (S^3, T({\psi(p)})) \quad and \quad \psi_q: (S^3, T(q)) \to (S^3, T({\psi(q)}))\] the associated homeomorphism of genus $g$ splittings.  Let $\gamma \subset T_2$ be an arc whose endpoints are $p$ and $q$, so $\psi(\gamma)$ is an arc whose endpoints are  $\psi(p)$ and $\psi(q)$.   Then 
\[ \psi_p \gamma^p_q \sim (\psi(\gamma))^{\psi(p)}_{\psi(q)} \psi_q   :T(q) \to T({\psi(p)}) \]
\end{lemma}

 \begin{figure}[ht!]
  \labellist
\small\hair 2pt
\pinlabel  $\frb_{\frc}@q$ at 140 145
\pinlabel  $\psi_q$ at 140 95
\pinlabel  $\psi_p$ at 470 95
\pinlabel $\gamma$ at 125 185
\pinlabel  $\psi(\gamma)$ at 165 55
\pinlabel  $(\psi(\gamma))^{\psi(p)}_{\psi(q)}$ at 300 40
\pinlabel  $\gamma^p_q$ at 290 170
\pinlabel  $p$ at 110 200
\pinlabel  $\frb_{\frc}@\psi(q)$ at 190 15
\pinlabel  $\frb_{\frc}@p$ at 400 210
\pinlabel  $\frb_{\frc}@\psi(p)$ at 520 82
\pinlabel  $T(\psi(p))$ at 430 15
\pinlabel  $T(q)$ at 200 140
\endlabellist
    \centering
    \includegraphics[scale=0.7]{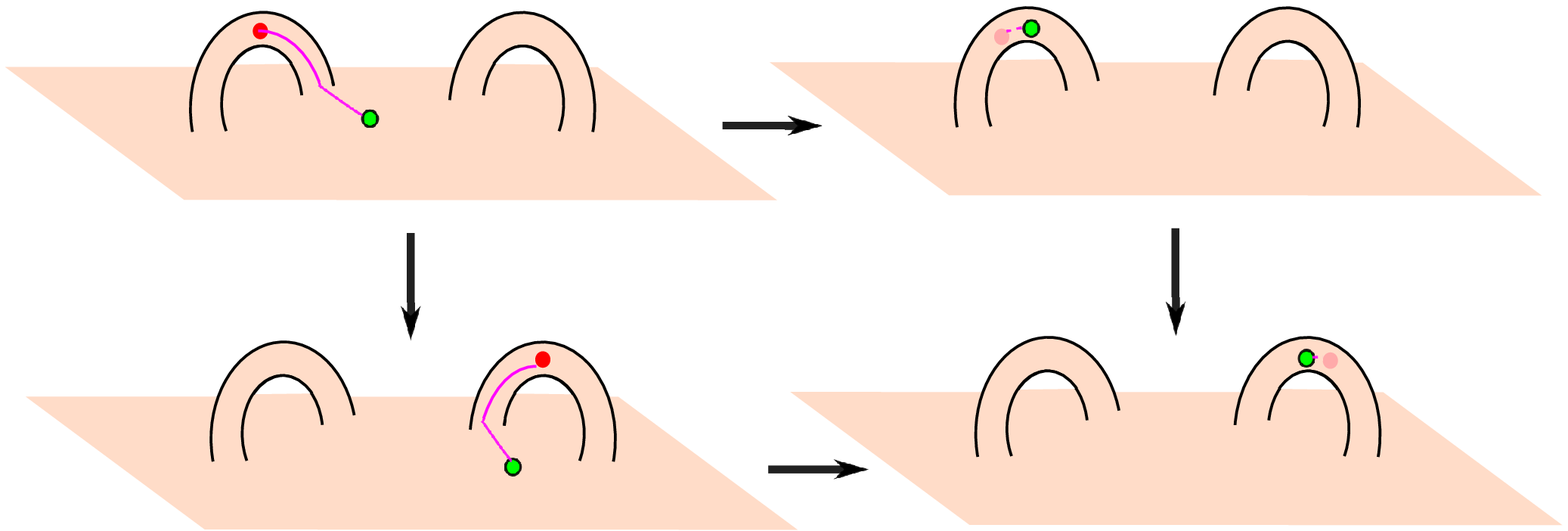}
   \caption{$\psi_p \gamma^p_q \sim (\psi(\gamma))^{\psi(p)}_{\psi(q)} \psi_q   :T(q) \to T({\psi(p)})$} \label{fig:homeocommute}
    \end{figure}

    \begin{proof}  
    Suppose first that $\gamma$ is an embedded arc.  Let $U$ be the ball neighborhood of $\gamma$ in $S^3$ in which the support of $\gamma^p_q$ lies, including all of the genus $g-2$ bubble $\frb_{\frc}$ placed at $q$.  Then $\psi_q(U)$ contains the support of $(\psi(\gamma))^{\psi(p)}_{\psi(q)}$ as well as the genus $g-2$ bubble $\frb_{\frc}$ placed at  the point $\psi(q)$ by $\psi$.  
    
    On $\psi_q(U)$ we may as well take $(\psi(\gamma))^{\psi(p)}_{\psi(q)}$ to be the topological conjugate (by $\psi_q|U$) of $\gamma^p_q$. In this case automatically  $\psi_p \gamma^p_q (x) = (\psi(\gamma))^{\psi(p)}_{\psi(q)} \psi_q (x)$ for $x \in U$.   On the other hand, for $x \in S^3 - U$, so $\psi(x) \in S^3 - \psi(U)$ 
    \[ \psi_p \gamma^p_q (x) = \psi(x) = (\psi(\gamma))^{\psi(p)}_{\psi(q)} \psi_q (x),\]
    completing the proof in this case.  
    
    In the general case we follow the logic of Lemma \ref{lemma:gammapq}(5).  For any (general position) arc $\alpha \subset T$ let $s(\alpha) \geq 1$ be the minimum number of subarcs into which $\alpha$ needs to be broken, so that each subarc is embedded.  We have just shown that the lemma is true if $\gamma$ is embedded, that is if $s(\gamma) = 1$.  If $s(\gamma) \geq 2$ suppose, inductively, the Lemma is known for all arcs $\alpha$ for which $s(\alpha) \leq s(\gamma) - 1$.  Then by definition there is a point $r \in \inter(\gamma)$ breaking $\gamma$ into the concatenation of arcs $\gamma_1, \gamma_2$ in which $s(\gamma_1) = 1$ and $s(\gamma_2) = s(\gamma) - 1$.  In particular, the Lemma is true for each of $\gamma_1$ and $\gamma_2$.  
Furthermore, by Lemma \ref{lemma:gammapq}, 
\[\gamma^p_q \sim (\gamma_2)^p_r (\gamma_1)^r_q \quad \text{and} \quad (\psi(\gamma_2))^{\psi(p)}_{\psi(r)}(\psi(\gamma_1))^{\psi(r)}_{\psi(q)} \sim
  (\psi(\gamma))^{\psi(p)}_{\psi(q)}.\]  Thus we have, applying the Lemma twice, 

   \begin{align*}
 \psi_p \gamma^p_q \sim 
 \psi_p \;((\gamma_2)^p_r (\gamma_1)^r_q ) &= 
 (\psi_p {\gamma_2}^p_r )\;(\gamma_1)^r_q  \sim 
( (\psi(\gamma_2))^{\psi(p)}_{\psi(r)}\psi_r) \;(\gamma_1)^r_q  =\\
 (\psi(\gamma_2))^{\psi(p)}_{\psi(r)}\;(\psi_r {\gamma_1}^r_q)  &\sim
(\psi(\gamma_2))^{\psi(p)}_{\psi(r)}(\psi(\gamma_1))^{\psi(r)}_{\psi(q)}\psi_q \sim
  (\psi(\gamma))^{\psi(p)}_{\psi(q)} \psi_q,
\end{align*}
    as required.
 \end{proof}
 
 The discussion above gives one way of moving from a homeomorphism $\psi: (S^3, T_2) \to (S^3, T_2)$ and a point $p \in T_2$ to a homeomorphism $(S^3, T(p)) \to (S^3, T(p))$:  choose an arc $\gamma \subset T_2$ whose endpoints are $p$ and $\psi^{-1}(p)$ and define
 \[ \psi_{\gamma} \equiv \psi_{\psi^{-1}(p)} \gamma_p^{\psi^{-1}(p)}:(S^3, T(p)) \to (S^3, T(p)).\] 
 Per Lemma \ref{lemma:homeocommute} this is equivalent to $[\psi(\gamma)]^p_{\psi(p)} \psi_p$.  
 
 Of course the homeomorphism $\phi_{\gamma}$ just defined requires a choice of $\gamma$, at least up to homotopy, per Lemma \ref{lemma:gammapq}(6).  We intend to show that, under restricted conditions, $\psi_{\gamma}$ can be made at least somewhat independent from the choice of $\gamma$.  To be specific, if $\gamma' \subset T_2$ is a different arc with endpoints $p$ and $\psi^{-1}(p)$ then, per  Lemma \ref{lemma:gammapq}(2),
 \[\psi_{\gamma} = \psi_{\psi^{-1}(p)} \gamma_p^{\psi^{-1}(p)} \sim\psi_{\psi^{-1}(p)} \;[(\gamma')_p^{\psi^{-1}(p)}(\gamma')^p_{\psi^{-1}(p)}]\;\gamma_p^{\psi^{-1}(p)}) \sim \psi _{\gamma'}[(\gamma')^p_{\psi^{-1}(p)}\gamma_p^{\psi^{-1}(p)})].\]

Per Lemma \ref{lemma:gammapq}(3,7), the last terms $(\gamma')^p_{\psi^{-1}(p)}\gamma_p^{\psi^{-1}(p)}$ describe a generic bubble move of $\frb_{\frc}$ around a loop in $T_2$.  If it were known that such a bubble move is Powell, then $\psi_{\gamma}$ and $\psi_{\gamma'}$ would be Powell equivalent, a potentially useful outcome.   In the next section we describe a special circumstance in which this is true.

\section{Towards a useful function $\Dih_6 \to \calP_{g}\bs\calG_g/\calP_g$}

Recall $\kappa \subset T_2$ is the union of the six circles that are the boundaries of the labeled meridians.  

\begin{defin} \label{defin:semistand}  Suppose $p \in T_2 - \kappa$ and  $h:(S^3, T_g) \to (S^3, T(p))$ is a homeomorphism of pairs so that the image of the standard orthogonal pair $(a_g, b_g)$ of meridian disks in the standard bubble $\frb_g$ is sent to a labeled orthogonal pair of disks $(a, b)$ in $(S^3, T(p))$.  Then $h$ defines a {\em semi-standard} structure on $T(p)$.  
\begin{itemize} 
\item The pair $(a, b)$ is the associated {\em standard pair} of labeled meridians, determining the {\em standard bubble} in $T(p)$.  $\bdd a \cup \bdd b \subset \kappa$ is the {\em standard pair $\kappa_s$} of circles in $\kappa$.

\item The orthogonal pair of labeled meridian disks $(x, y)$ in $(S^3, T(p))$ that is disjoint from $(a, b)$ is the {\em forbidden pair} of labeled meridians. $\bdd x \cup \bdd y \subset \kappa$ is the forbidden pair $\kappa_f$ of circles in $\kappa$.

\item Let $\calG(p)$ denote the topological conjugate of $\calG_g$ by $h$, for example: 

\item The Powell subgroup $\calP(p) \subset \calG(p)$ is the topological conjugate of $\calP_g \subset \calG_g$ by $h$.  That is,
 \[\psi: (S^3, T(p)) \to (S^3, T(p)) \in \calP(p) \iff h^{-1}\psi h: (S^3, T_g) \to (S^3, T_g) \in \calP_g.\]
 \end{itemize}
\end{defin}

\begin{prop} \label{prop:G/P}  Suppose $p \in T_2 - \kappa$, and $T(p)$ has a semi-standard structure, with $\kappa_f$ the forbidden pair of circles in $\kappa$.  Suppose further that $\psi \in \Dih_6$ and $\gamma_1, \gamma_2$ are arcs in $T_2 - \kappa_f$, each with one end-point on $p$ and the other end point on $\psi^{-1}(p)$.  Then 
\[\psi_{\gamma_1} \calP(p) = \psi_{\gamma_2}\calP(p)\]
\end{prop}  

\begin{proof} As noted at the end of Section \ref{sect:firstobs},
it suffices to show that ${\gamma_1}^p_{\psi^{-1}(p)} {\gamma_2}_p^{\psi^{-1}(p)} \in \calP(p)$ and  
Lemma \ref{lemma:gammapq}(3,7) shows that ${\gamma_1}^p_{\psi^{-1}(p)} {\gamma_2}_p^{\psi^{-1}(p)}$ is a bubble move of $\frb_{\frc}$ around a loop in $T_2 - \kappa_f$.  On the other hand, $T_2 - \kappa_f$ is a punctured torus with spine the standard pair of circles $\kappa_s$, so $\pi_1(T_2 - \kappa_f) \cong \pi_1(\kappa_s) \cong \zed * \zed$, generated by $\pi_1(\bdd a)$ and $\pi_1(\bdd b)$.  Hence ${\gamma_1}^p_{\psi^{-1}(p)} {\gamma_2}_p^{\psi^{-1}(p)}$ is homotopic to a sequence of bubble moves, each one around either $\bdd a$ or $\bdd b$.  According to Corollary \ref{cor:eyeglass3o} each such bubble move is a Powell move.
\end{proof}

\begin{defin} \label{defin:thetap}  Suppose $p \in T_2 - \kappa$, and $T(p)$ has a semi-standard structure, with $\kappa_f$ the forbidden pair of circles in $\kappa$. 

Let $\Theta_p: \Dih_6 \to \calP(p) \bs\calG(p)/\calP(p)$ be the function given by $\Theta_p(\psi) = \calP(p) \psi_{\psi^{-1}(p)} {\gamma}_p^{\psi^{-1}(p)} \calP(p)$, for $\gamma$ any path in $T_2 - \kappa_f$ that has end points at $p$ and $\psi^{-1}(p)$.  
\end{defin}

$\Theta_p$ is well-defined by Proposition \ref{prop:G/P}.  (Note that it is well-defined even as a function to $\calG(p)/\calP(p)$; the need to retreat to double cosets is explained before Corollary \ref{cor:PpPq}.)

\begin{lemma} \label{lemma:fixedp} If $p \in T_2 - \kappa$ is a fixed point for $\psi \in \Dih_6$ then $\Theta_p(\psi) = \calP(p) \psi_p  \calP(p)$.
\end{lemma}

\begin{proof}  Since $p$ is a fixed point, $\psi^{-1}(p) = p$ and we can take ${\gamma}_p^{\psi^{-1}(p)}$ to be the identity.  
\end{proof}

We do not know that $\calP(p)$ is a normal subgroup of $\calG(p)$, so we cannot say that $\Theta_p$ is a homomorphism.  But we do have this:

\begin{lemma}   \label{lemma:thetaprod} Suppose, for $p \in T_2 - \kappa$, $\Theta_p (\psi_1) = \calP(p)$ and $\Theta_p (\psi_2) = \calP(p)$.  Then $\Theta_p (\psi_2 \psi_1) = \calP(p)$.
\end{lemma}

\begin{proof}  The hypothesis implies that for any paths $\gamma_i \subset T_2 - \kappa_f, i = 1, 2$  with ends at $p$ and $\psi_i^{-1}(p)$ we have $(\psi_i)_{\psi_i^{-1}(p)}\; [\gamma_i]_p^{\psi_i^{-1}(p)} \in \calP(p)$.  This implies
\begin{equation} (\psi_2)_{\psi_2^{-1}(p)}\; [\gamma_2]_p^{\psi_2^{-1}(p)} (\psi_1)_{\psi_1^{-1}(p)}\; [\gamma_1]_p^{\psi_1^{-1}(p)} \in \calP(p).  
\end{equation}

According to Lemma \ref{lemma:homeocommute}, the middle terms \[[\gamma_2]_p^{\psi_2^{-1}(p)} (\psi_1)_{\psi_1^{-1}(p)} \sim (\psi_1)_{\psi_1^{-1}\psi_2^{-1}(p)}\; [\psi_1^{-1}(\gamma_2)]_{\psi_1^{-1}(p)}^{\psi_1^{-1}\psi_2^{-1}(p)}\] so  (1) becomes
\begin{equation} 
(\psi_2)_{\psi_2^{-1}(p)}\;(\psi_1)_{\psi_1^{-1}\psi_2^{-1}(p)}\; [\psi_1^{-1}(\gamma_2)]_{\psi_1^{-1}(p)}^{\psi_1^{-1}\psi_2^{-1}(p)}[\gamma_1]_p^{\psi_1^{-1}(p)} \in \calP(p).  
\end{equation}
See Figure \ref{fig:thetaprod}.  

 \begin{figure}[ht!]
  \labellist
\small\hair 2pt
\pinlabel  $\gamma_1$ at 240 450
\pinlabel  $\gamma_2$ at 280 450
\pinlabel  $p$ at 260 415
\pinlabel  $\psi_1^{-1}(\gamma_2)$ at 160 440
\pinlabel  $\psi_1^{-1}(p)$ at 210 475
\pinlabel  $\psi_2^{-1}(p)$ at 310 475
\pinlabel  $[\gamma_1]^{\psi^{-1}(p)}_p$ at 300 370
\pinlabel  $[\psi_1^{-1}(\gamma_2)]_{\psi_1^{-1}(p)}^{\psi_1^{-1}\psi_2^{-1}(p)}$ at 160 260
\pinlabel  $\psi_1^{-1}\psi_2^{-1}(p)$ at 285 290
\pinlabel  $(\psi_1)_{\psi_1^{-1}(p)}$ at 360 260
\pinlabel  $(\psi_1)_{\psi_1^{-1}\psi_2^{-1}(p)}$ at 160 140
\pinlabel  $[\gamma_2]_p^{\psi_2^{-1}(p)}$ at 430 140
\pinlabel  $(\psi_2)_{\psi_2^{-1}(p)}$ at 280 25
\endlabellist
    \centering
    \includegraphics[scale=0.7]{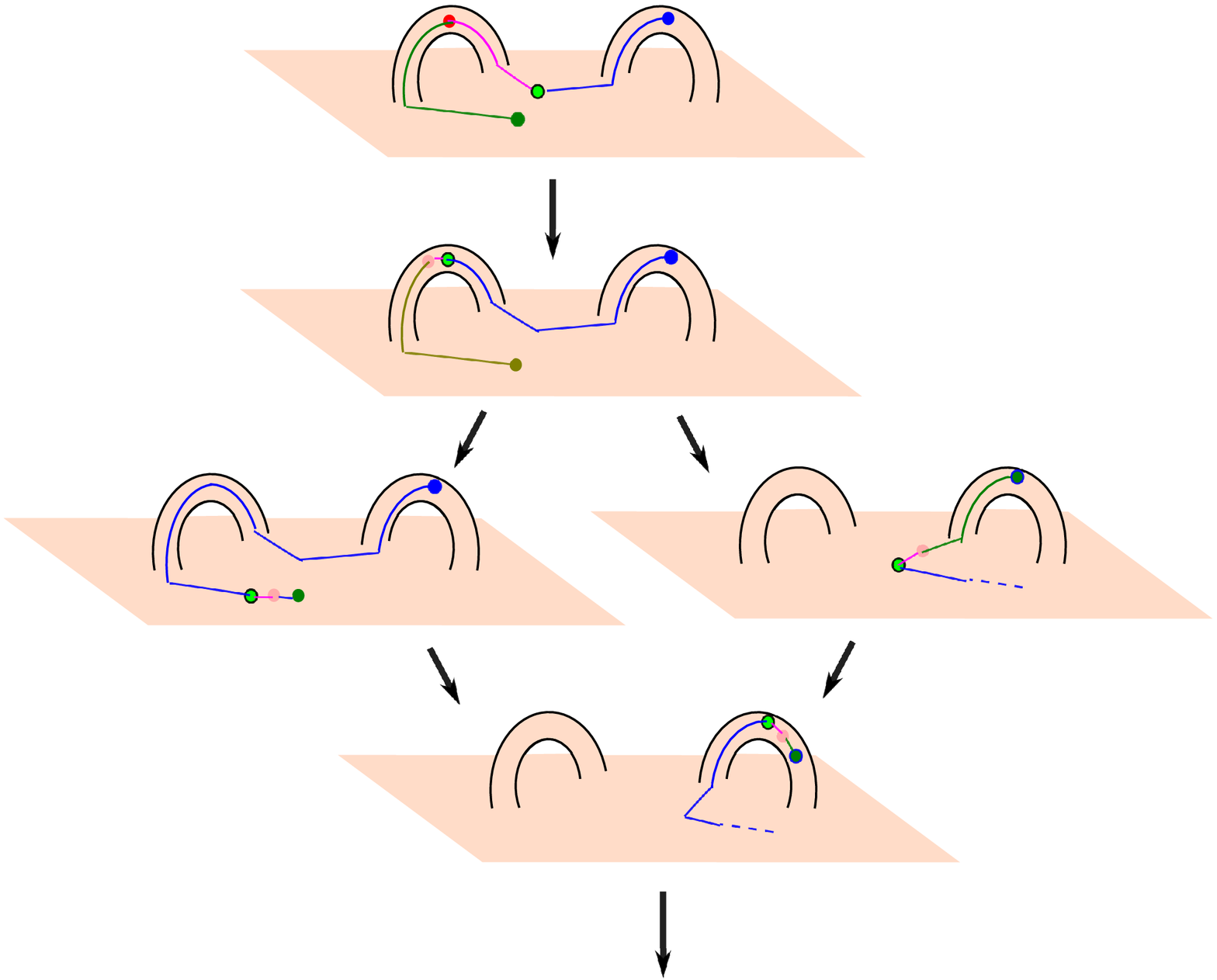}
   \caption{} \label{fig:thetaprod}
    \end{figure}

The arcs $\gamma_1$ and $\psi_1^{-1}(\gamma_2)$ have a common endpoint at $\psi_1^{-1}(p)$ so $\gamma_3 = \gamma_1 \cup \psi_1^{-1}(\gamma_2)$ is an arc with end points at $p$ and ${\psi_1^{-1}\psi_2^{-1}(p)}$.  Then by Lemma \ref{lemma:gammapq} (2) becomes
$$ (\psi_2\;\psi_1)_{\psi_1^{-1}\psi_2^{-1}(p)}\;[\gamma_3]_p^{\psi_1^{-1}\psi_2^{-1}(p)} \in \calP(p).$$
This expression in turn represents $\Theta_p(\psi_2\psi_1)$, {\em so long as $\gamma_3 \subset T_2 - \kappa_f$}.  This is true if and only if $\psi_1^{-1}(\gamma_2) \subset T_2 - \kappa_f$, that is $\gamma_2 \subset T_2 - \psi_1(\kappa_f)$.  This is easily arranged since, by Lemma \ref{lemma:Kdef}, $T_2 - (\kappa_f \cup  \psi_1(\kappa_f))$ is connected.
\end{proof}

\begin{cor} \label{cor:Kpexist}  $\Theta_p^{-1}(\calP(p))$ is a subgroup $K(p) \subset \Dih_6$. \qed
\end{cor}

Suppose $p, q \in T_2 - \kappa$ and the arc $\alpha \subset T_2 - \kappa_s$ has end points at $p$ and $q$.  Then, since $\alpha$ is disjoint from $\kappa_s$, $\hat{h} = \alpha^q_p h: (S^3, T_g) \to (S^3, T(q))$ defines a semi-standard structure on $T(q)$, one in which $(a, b)$ is still the associated standard pair.  We have:

\begin{lemma}  \label{lemma:Kq} For the semi-standard structure on $T(q)$ given by  $\hat{h}$, $$\calP(q) = \alpha_p^q\calP(p)\alpha_q^p.$$  
\end{lemma}

\begin{proof}  For $\psi: (S^3, T(q)) \to (S^3, T(q))$
 \[\psi \in \calP(q) \iff \hat{h}^{-1}\psi \hat{h} = h^{-1}\alpha^p_q \psi \alpha^q_p h \in \calP_g \iff \alpha^p_q \psi \alpha^q_p \in \calP(p).\]
\end{proof}

We can go further, and define $\Theta_q: \Dih_6 \to \calP(q) \bs \calG(q)/\calP(q)$ as in Definition \ref{defin:thetap}, but there is a technical issue in defining $\Theta_q$ in a  way that is naturally related to $\Theta_p$:  For each $\psi \in \Dih_6$, defining $\Theta_q(\psi)$ requires a path in $T_2$ with end points at $q$ and $\psi^{-1}(q)$ that is disjoint from $\kappa_f$.   $\Theta_p(\psi)$ is defined via a path $\gamma$ in $T_2 - \kappa_f$ with end points at $p$ and $\psi^{-1}(p)$, so, in defining $\Theta_q(\psi)$ a natural choice would be to use the path $\psi^{-1}(\alpha)_{\psi^{-1}(p)}^{\psi^{-1}(q)} \gamma_p^{\psi^{-1}(p)} \alpha_q^p$, which has end points at $q$ and $\psi^{-1}(q)$.  

Here is some of what we know about the three arcs used in this definition:
\begin{itemize}
\item Per Definition \ref{defin:thetap} $\gamma$ is disjoint from $\kappa_f$
\item Per the preliminary remarks to Lemma \ref{lemma:Kq}, $\alpha$ is disjoint from $\kappa_s$
\end{itemize}

To use the arc $\psi^{-1}(\alpha)_{\psi^{-1}(p)}^{\psi^{-1}(q)} \gamma_p^{\psi^{-1}(p)} \alpha_q^p$ in defining $\Theta_q(\psi)$ we need, per Definition \ref{defin:thetap}, the arc to be disjoint from $\kappa_f$.  Of the three segments involved, this is automatic for $\gamma$.  To arrange it also for $\alpha$ we need only add that requirement, that is, choose $\alpha$ to be disjoint from both $\kappa_s$ and $\kappa_f$.  This is possible per Lemma \ref{lemma:Kdef}(4) (or simple observation).  But once this is done, we do not know that the third segment, namely $\psi^{-1}(\alpha)$, is disjoint from $\kappa_f$; for this we would need to add a third requirement to $\alpha$, namely that it is also disjoint from $\psi(\kappa_f)$ and, depending on $\psi$, adding this third requirement may not be possible, see Figure \ref{fig:leftPnec}.

\begin{figure}[ht!]
  \labellist
\small\hair 2pt
\pinlabel  $\aaa$ at 155 80
\pinlabel  $\psi(\kappa_f)$ at 115 48
\pinlabel  $\kappa_s$ at 30 76
\pinlabel  $\kappa_f$ at 245 76
\endlabellist
    \centering
    \includegraphics[scale=0.7]{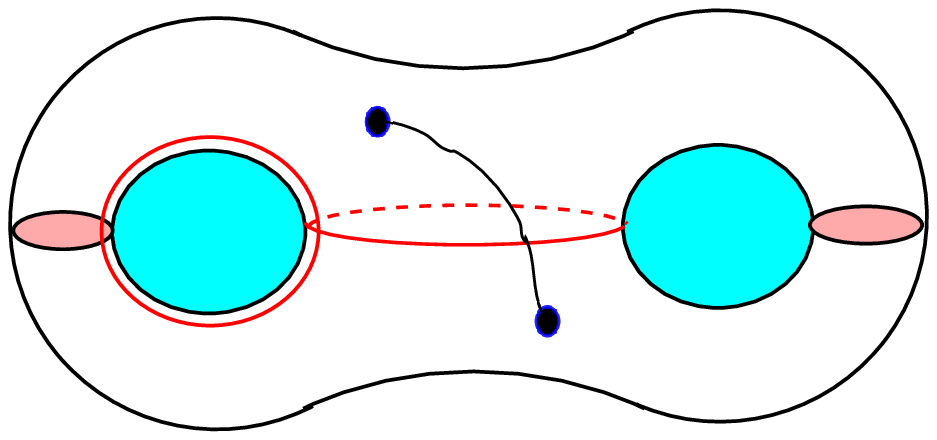}
   \caption{} \label{fig:leftPnec}
    \end{figure}

One approach to fixing this problem (an approach that requires retreating to the use of double cosets) is to replace  $\psi^{-1}(\alpha)$ with yet a third arc, say $\beta$, with the same end points, chosen to be disjoint from $\kappa_f$.  Moreover it will be useful if $\beta$ is also disjoint from $\psi^{-1}(\kappa_f)$, as we can do per  Lemma \ref{lemma:Kdef}(4).  

Using the arc $\beta_{\psi^{-1}(p)}^{\psi^{-1}(q)} \gamma_p^{\psi^{-1}(p)} \alpha_q^p$ to define  $\Theta_q: \Dih_6 \to \calP(q) \bs\calG(q)/\calP(q)$ and applying Lemma \ref{lemma:homeocommute}, we have 
\[ \Theta_q(\psi)=\calP(q)\psi_{\psi^{-1}(q)}  \beta_{\psi^{-1}(p)}^{\psi^{-1}(q)} \gamma_p^{\psi^{-1}(p)} \alpha_q^p\calP(q)= \calP(q)[\psi(\beta)]_p^q \psi_{\psi^{-1}(p)}  \gamma_p^{\psi^{-1}(p)} \alpha_q^p\calP(q).\] 

The last term can be written $\calP(q)[\psi(\beta)]_p^q(\alpha_q^p \alpha_p^q)\psi_{\psi^{-1}(p)}  \gamma_p^{\psi^{-1}(p)} \alpha_q^p\calP(q)$.  Observe that the pair $[\psi(\beta)]_p^q\alpha_q^p: T(q) \to T(q)$ is a bubble move around a path $\psi(\beta) \cup \alpha$ that is disjoint from $\kappa_f$, so per Corollary \ref{cor:eyeglass3o} and the proof of Proposition \ref{prop:G/P}, it belongs to $\calP(q)$, and so can be absorbed into $\calP(q)$ on the left.  Thus, using Lemma \ref{lemma:Kq}, this can also be written 
\[\calP(q)\alpha_p^q\psi_{\psi^{-1}(p)}  \gamma_p^{\psi^{-1}(p)} \alpha_q^p\calP(q) = \alpha_p^q\calP(p)\psi_{\psi^{-1}(p)}  \gamma_p^{\psi^{-1}(p)}\calP(p) \alpha_q^p = \alpha_p^q\Theta_p(\psi) \alpha_q^p. \]
Summing up, we have the simple relation:

\begin{cor} \label{cor:PpPq} For, $p, q \in T_2 - \kappa$, $\Theta_q(\psi) = \alpha_p^q \Theta_p(\psi) \alpha_q^p$ so in particular 
$\Theta_q(\psi) \in \calP(q) \iff \Theta_p(\psi) \in \calP(p)$ and $K(p) = K(q) \subset \Dih_6$.  \qed
\end{cor}

We summarize the above discussion, and extend it slightly, in the following Propositions.  Recall that in the discussion above $p, q \in T_2 - \kappa$; below we will only require that $p, q \in T_2 - \kappa_f$, where $\kappa_f \subset \kappa$ is a specified forbidden pair of labeled meridians.  That is, in the Proposition and proof below we reprise the discussion above, but also allow the possibility that one or both of $p, q \in \kappa - \kappa_f$.

\begin{prop}. \label{prop:thetastar}  Let $\kappa_0 \subset \kappa \subset T_2$ be a pair of circles that bound an orthogonal pair of labeled meridian disks.  One can define, for each $p \in T_2 - \kappa_0$, a collection of homeomorphisms $\calH_p$ from $(S^3, T_g)$ to $(S^3, T(p))$ with the following properties:
\begin{enumerate}
\item For each $p \notin \kappa$, there is a subset $\calH_p^0 \subset \calH_p$ so that each $h^0 \in \calH_p^0$ gives a semi-standard structure on $T(p)$ in which $\kappa_0$ is the forbidden pair of labeled meridians.
\item For each $p \in T^2 - \kappa_0$ and $h_p \in \calH_p$, $\calP_{h_p} = \{\psi \subset \calG(p) | h_p^{-1} \psi h_p \in \calP_g\}$ is independent of $h_p$.  That is, for any pair $h_1, h_2 \in \calH_p$, $\calP_{h_1} = \calP_{h_2}$.  This subgroup of $\calG(p)$ is denoted $\calP(p)$.  
\item For each $p, q \in T_2 - \kappa_0$ and any arc $\gamma \subset T_2 - \kappa_0$, $\calP(q) = \gamma_p^q \calP(p) \gamma_q^p$.
\end{enumerate}
\end{prop}

\begin{proof}  Pick a fixed base point $\star \in T_2 - \kappa$ and a homeomorphism of pairs $h:(S^3, T_g) \to (S^3, T(\star))$ that gives a semi-standard structure on $T(\star)$ in which $\kappa_0$ is the forbidden pair of labeled meridians.  Given any point $p \in T_2 - \kappa_0$ define $\calH_p$ to be the set of all homeomorphisms $ \gamma_{\star}^p h : (S^3, T_g) \to (S^3, T(p))$, where $\gamma \subset T_2 - \kappa_0$ is a path with one end at $\star$ and one end at $p$.  If $p \notin \kappa$ then define $\calH^0_p$ to be the subcollection in which $\gamma \subset T_2 - (\kappa_0 \cup \kappa_s)$.  

To prove (1) observe that when $\gamma \subset T_2 - (\kappa_0 \cup \kappa_s)$ the homeomorphism $\gamma_{\star}^p$ 
is the identity on $(\kappa_0 \cup \kappa_s)$, so $ \gamma_{\star}^p h $ and $h$ coincide over $(\kappa_0 \cup \kappa_s)$.  It follows from Definition \ref{defin:semistand} that $h$ gives a semi-standard structure on $T(\star)$ in which $\kappa_0$ is the forbidden pair of labeled meridians, if and only if the same is true for $T(p)$ and $ \gamma_{\star}^p h $.

To prove (2) choose arcs $\alpha, \beta \subset T_2 - \kappa_0$ so that $h_1 = \alpha_{\star}^p h$ and $h_2 = \beta_{\star}^p h$.  
As shown in the proof of Proposition \ref{prop:G/P}, $\beta_p^{\star}\alpha_{\star}^p \in \calP(\star)$ so, per Definition \ref{defin:semistand}
$h_2^{-1}h_1 = h^{-1}\beta_p^{\star}\alpha_{\star}^p h \in \calP_g$.  We then have
\[ 
h_1^{-1} \psi h_1 \in \calP_g \iff  (h_2^{-1}h_1) h_1^{-1} \psi h_1 (h_1^{-1} h_2) = h_2^{-1} \psi h_2 \in \calP_g 
\]
It follows that $\psi \in \calP_{h_1} \iff \psi \in \calP_{h_2}$, as required.  

For (3), choose arcs $\alpha, \beta \subset T_2 - \kappa_0$ so that $h_{\alpha}= \alpha_{\star}^p h \in \calH_p$ and $h_{\beta} = \beta_{\star}^q h \in \calH_q$. Then the curve $\alpha \cup \gamma \subset T_2 - \kappa_0$ has ends on $q$ and $\star$ hence $(\alpha \cup \gamma)_{\star}^q h \in \calH_q$.  On the other hand, Lemma \ref{lemma:gammapq} says $(\alpha \cup \gamma)_{\star}^q \sim \gamma_p^q \alpha_{\star}^p$.   Thus 

\[\calP(q) = (\alpha \cup \gamma)_{\star}^q h \calP_g h^{-1} (\alpha \cup \gamma)^{\star}_q=  \gamma_p^q \alpha_{\star}^p h \calP_g   h^{-1} \alpha^{\star}_p \gamma_q^p  = \gamma_p^q \calP(p) \gamma_q^p\] as required.
\end{proof}

\begin{prop} \label{prop:dihedral}  Let $\kappa_0 \subset \kappa \subset T_2$ be a pair of circles that bound an orthogonal pair of labeled meridian disks.  One can define, for each $p \in T_2 - \kappa_0$, a function $\Theta_p: \Dih_6 \to \calP(p) \bs \calG(p)/\calP(p)$ with the following properties:
\begin{enumerate}
\item  When $p \notin \kappa$, the definition coincides with that in Definition \ref{defin:thetap}.
\item  For any $p, q \in T_2 - \kappa_0$ and path $\alpha \subset T^2 - \kappa_0$ with end points at $p$ and $q$, $\Theta_q(\psi) = \alpha_p^q \Theta_p(\psi) \alpha_q^p$ so in particular 
$\Theta_q(\psi) = \calP(q) \iff \Theta_p(\psi) = \calP(p)$.
\item  Suppose $\Theta_p (\psi_1) = \calP(p)$ and $\Theta_p (\psi_2) = \calP(p)$.  Then $\Theta_p (\psi_2 \psi_1) = \calP(p)$.
\item Suppose $\psi \in \Dih_6$ has a fixed point $p \notin \kappa_0$.  Then $\Theta_p(\psi) = \calP(p)\psi_p \calP(p)$.   
\end{enumerate}
\end{prop}

\begin{proof} As in the proof of Proposition \ref{prop:thetastar}, pick a fixed base point $\star \in T_2 - \kappa$ and a homeomorphism of pairs $h:(S^3, T_g) \to (S^3, T(\star))$ that gives a semi-standard structure on $T(\star)$ in which $\kappa_0$ is the forbidden pair of labeled meridians.   The function $\Theta_{\star}$ is given in Definition \ref{defin:thetap}.  

For $p \in T^2 - \kappa_0$, choose an arc $\alpha \subset T - \kappa_0$ with end points on $p$ and $\star$ and define $\Theta_p(\psi) =  \alpha_{\star}^p \Theta_{\star}(\psi) \alpha_p^{\star}$.  $\Theta_p(\psi)$ is well-defined, for if $\beta \subset T - \kappa_0$ is another such arc then 
\[ \beta_{\star}^p \theta_{\star}(\psi) \beta_p^{\star} = (\alpha_{\star}^p\alpha^{\star}_p)\beta_{\star}^p \theta_{\star}(\psi) \beta_p^{\star}(\alpha_{\star}^p\alpha^{\star}_p) = \alpha_{\star}^p(\alpha^{\star}_p\beta_{\star}^p) \theta_{\star}(\psi)( \beta_p^{\star}\alpha_{\star}^p)\alpha^{\star}_p = \alpha_{\star}^p \theta_{\star}(\psi)\alpha^{\star}_p\]
The last equality follows because, as argued in the proof of Proposition \ref{prop:G/P}.  the bubble move $\alpha^{\star}_p\beta_{\star}^p \in \calP(\star)$.  

 (1) follows immediately from Corollary \ref{cor:PpPq}.  

The proof of (2) is almost as easy: Suppose $\beta \subset T^2 - \kappa_0$ is a path with end points at $\star$ and $p$, so $\alpha \cup \beta$ has ends at $\star$ and $q$.  Then 
\[ \Theta_q(\psi) = (\alpha \cup \beta)_{\star}^q \Theta_{\star}(\psi) (\alpha \cup \beta)^{\star}_q  =    \alpha_p^q \beta_{\star}^p \Theta_{\star}(\psi) \beta^{\star}_p \alpha_q^p  = \alpha_p^q \Theta_p(\psi) \alpha_q^p \]

Observe that (3) is true for $\Theta_{\star}$ by Lemma \ref{lemma:thetaprod}.  Suppose then that $p \in T_2 - \kappa_0$ and $\Theta_p(\psi_1) = \Theta_p(\psi_2) = \calP(p)$.  Then by (2)
for $i = 1, 2$, $\Theta_{\star}(\psi_i) = \calP(\star)$.  
\[\Theta_p(\psi_2\psi_1) =  \alpha_{\star}^p \Theta_{\star}(\psi_2 \psi_1) \alpha^{\star}_p = \alpha_{\star}^p \calP(\star) \alpha^{\star}_p  = \calP(p),\]
as required.  (The last equality from Proposition \ref{prop:thetastar}(3).)

Claim (4) is true for $p \notin \kappa$ by Lemma \ref{lemma:fixedp}, but our argument will extend to the case $p \in \kappa - \kappa_0$.  Choose $\alpha$ above so that it disjoint not only from $\kappa_0$ but also from $\psi(\kappa_0)$.  This is possible by Lemma \ref{lemma:Kdef}(4).  Since $p$ is a fixed point of $\psi$,  $\gamma = \alpha \cup \psi^{-1}(\alpha)$ is an arc whose endpoints are $\star$ and $\psi^{-1}(\star)$.  Moreover, by choice of $\alpha$, $\gamma$ is disjoint from $\kappa_0$.  Then according to Definition \ref{defin:thetap} 
\[\Theta_{\star}(\psi) = \calP(\star) \psi_{\psi^{-1}(\star)} \gamma_{\star}^{\psi^{-1}(\star)} \calP(\star) = \calP(\star) \psi_{\psi^{-1}(\star)} [\psi^{-1}(\alpha)]_p^{\psi^{-1}(\star)} \alpha_{\star}^p \calP(\star) \]
and so, by  Lemma \ref{lemma:homeocommute}, 
\[\Theta_{\star}(\psi) =  \calP(\star) \alpha_p^{\star} \psi_p \alpha_{\star}^p \calP(\star)  \] 

Hence from Proposition \ref{prop:thetastar}(3)
\[\Theta_p(\psi) = \alpha_{\star}^p \Theta_{\star}(\psi) \alpha_p^{\star} = \alpha_{\star}^p \calP(\star) \alpha_p^{\star} \psi_p \alpha_{\star}^p \calP(\star)  \alpha_p^{\star} = \calP(p)\psi_p \calP(p)\]
as required.
\end{proof}

\section{Two examples}

We present two examples, each involving a stabilizer in $\Dih_6$ of one of the two standard labeled meridians, so we will use the planar viewpoint, that of  Figure \ref{fig:Kdef2}.  They use the same construction, but with a different choice of positioning for the pair of labeled meridians that are standard.  For that reason, it will be helpful to refer to the labeled meridians as they appear in the planar viewpoint without reference to the original labeling.  To that end, we will use in the argument notation from Figure \ref{fig:genswitch2}, where two of the three $B$-meridians are shown on left and right respectively as $w_\ell, w_r$ and two nearly parallel copies of the boundary of the third $B$-meridian are shown as curves $\beta_\pm$.  Similarly, two of the three $A$-meridians are shown on left and right respectively as $u_\ell, u_r$ and two parallel copies of the boundary of the third $A$-meridian are shown as curves $\alpha_\pm$.  

In Figure \ref{fig:genswitch2}, the genus $g-2$ bubble $\frb_{\frc}$ is placed between the curves $\bbb_{\pm}$; $\bbb_+$ lies on the front half of the figure and $\bbb_-$ is shown with dashes, for it lies on the back face.  The vertical arc $\lambda_+$ has ends at $\aaa_+$ to $\bbb_+$ and symmetrically  $\lambda_-$ has ends on $\aaa_-$ and $\bbb_-$.  This will be the position of $\frb_{\frc}$ in the first example; in the second, $\frb_{\frc}$ will be placed between the curves $\alpha_{\pm}$.

\begin{figure}[ht!]
  \labellist
\small\hair 2pt
\pinlabel  $\aaa_+$ at 165 80
\pinlabel  $\aaa_-$ at 115 48
\pinlabel  $\lambda_-$ at 145 40
\pinlabel  $\frb_{\frc}$ at 140 5
\pinlabel  $\lambda_+$ at 130 90
\pinlabel  $\bbb_-$ at 10 100
\pinlabel  $\bbb_+$ at 40 90
\pinlabel  $u_\ell$ at -5 65
\pinlabel  $u_r$ at 285 65
\pinlabel  $w_\ell$ at 65 65
\pinlabel  $w_r$ at 215 65
\endlabellist
    \centering
    \includegraphics[scale=0.7]{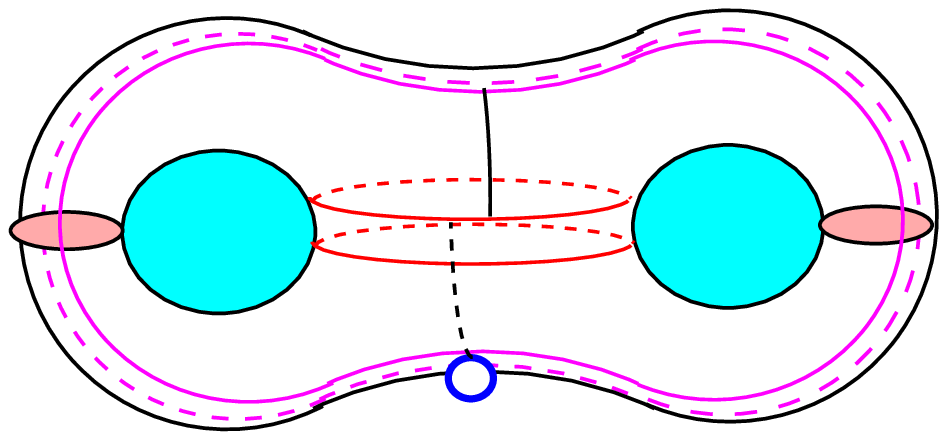}
   \caption{} \label{fig:genswitch2}
    \end{figure}

It will be useful to depict how the boundaries of the various disks appear on the $4$-punctured sphere $P \subset T_2$ that lies between the four circles $\bbb_{\pm}$ and $\aaa_{\pm}$: See the upper left panel of Figure \ref{fig:genswitch3}.  Slopes in the figure have been chosen so that $\lambda_{\pm}$ are vertical arcs (slope $\infty$) on the two sides, $\bdd w_r$ is the horizontal arc (slope $0$) at the top, and $\bdd u_{\ell}$ is the horizontal arc at the bottom.

\begin{figure}[ht!]
  \labellist
\small\hair 2pt
\pinlabel  $\aaa_+$ at 200 300
\pinlabel  $\aaa_-$ at 380 300
\pinlabel  $\aaa_+$ at -7 300
\pinlabel  $\aaa_-$ at 170 300
\pinlabel  $\aaa_+$ at 285 110
\pinlabel  $\aaa_-$ at 100 110
\pinlabel  $\lambda_-$ at 145 250
\pinlabel  $\lambda_+$ at 10 250
\pinlabel  $\bbb_-$ at 170 190
\pinlabel  $\bbb_+$ at -5 190
\pinlabel  $\bdd w_\ell$ at 35 250
\pinlabel  $\bdd w_r$ at 80 290
\pinlabel  $\bdd u_\ell$ at 80 200
\pinlabel  $\bdd u_r$ at 90 260
\pinlabel  $\epsilon$ at 80 240
\pinlabel  $\bdd u_\ell$ at 280 250
\pinlabel  $\bdd w_r$ at 335 250
\pinlabel  $\bdd w_\ell$ at 290 290
\pinlabel  $\bdd u_r$ at 290 205
\pinlabel  $\bdd w_r$ at 150 60
\pinlabel  $\bdd w_\ell$ at 190 105
\pinlabel  $\bdd u_r$ at 190 15
\pinlabel  $\bdd u_\ell$ at 200 60
\pinlabel  $\omega$ at 300 240
\pinlabel  eyeglass$\;\phi_\epsilon$ at 190 270
\pinlabel  {flip $\phi_{\aaa}$}  at 290 150
\pinlabel  $\rho_y$ at 90 150
\endlabellist
    \centering
    \includegraphics[scale=0.8]{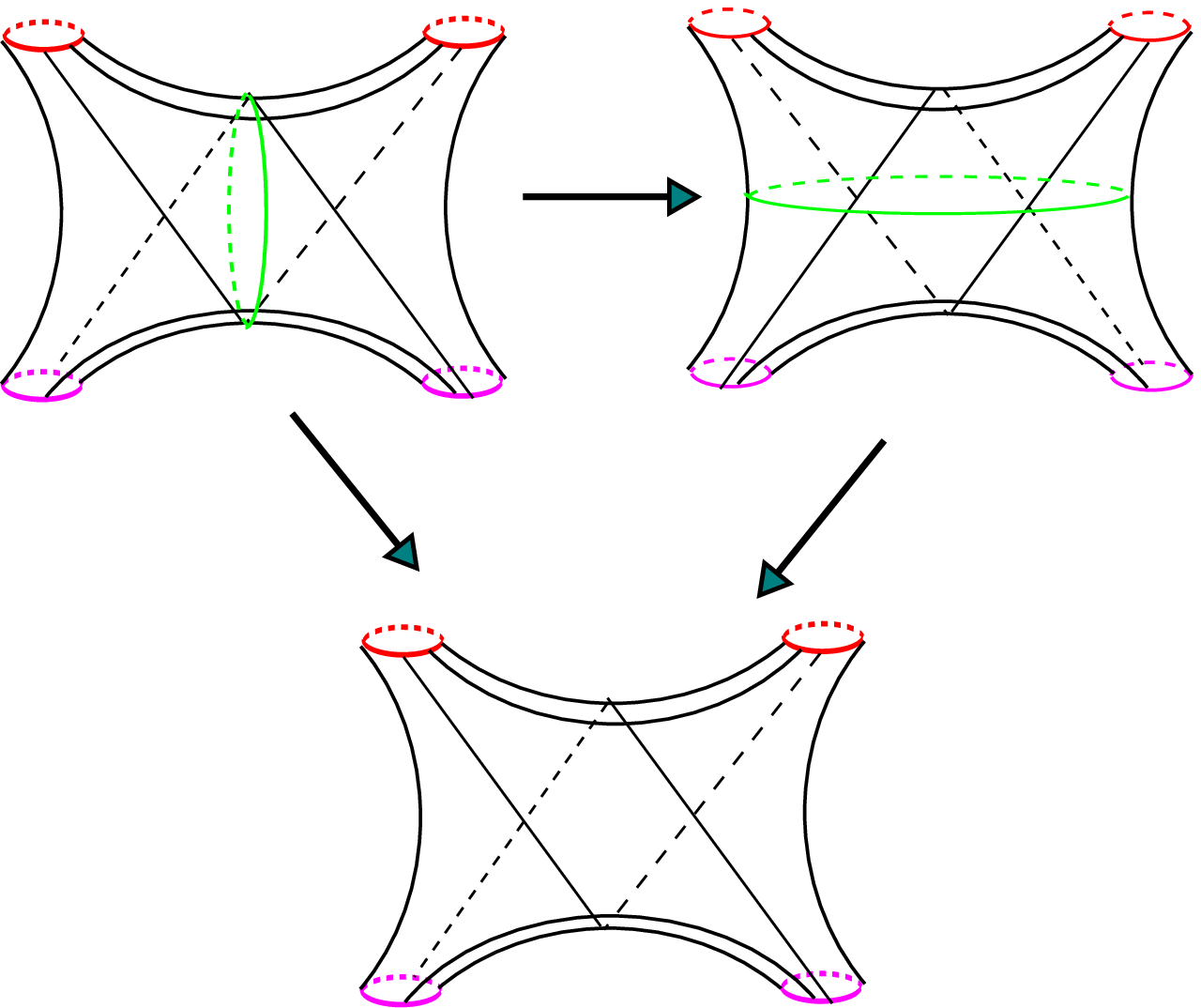}
   \caption{} \label{fig:genswitch3}
    \end{figure}

The green vertical circle $\epsilon$ in the upper left panel of figure \ref{fig:genswitch3} is parallel to the boundary of the eyeglass whose lenses are bounded by $\aaa_-, \bbb_-$ and whose bridge is $\lambda_-$.  Thus a corresponding eyeglass twist $\phi_{\eee}$ can be viewed in $P$ as a full Dehn twist around $\epsilon$; the images of the four circles $\bdd u_r, \bdd w_r, \bdd u_\ell, \bdd w_\ell$ are shown in the upper right panel.    

Also in that panel is a horizontal circle $\omega$ which, in $T_2$, cuts off a punctured torus containing $\alpha_{\pm}$ and $\bdd w_r$ (onto which $\phi_{\epsilon}$ has moved $\bdd w_\ell$).  Then $\omega$ bounds a disk in both $A$ and $B$ cutting off a genus 1 bubble, with $\alpha$ as a meridian; a flip $\phi_{\aaa}$ in this bubble is represented in $P$ by a half-twist along $\omega$ that switches $\aaa_{\pm}$.  The resulting images of the four circles $\bdd u_r, \bdd w_r, \bdd u_\ell, \bdd w_\ell$ after the eyeglass move followed by the flip is shown in the bottom panel.  
This is the same positioning of the four circles (as well as $\alpha$ and $\beta$) as one would obtain by the $\pi$-rotation $\rho_y$.   

(The remarks above on slope reflect these matrix observations: 
\begin{equation*}
\phi_{\epsilon} =\text{twist on }\epsilon = 
\begin{pmatrix}
1 & 2  \\
0 & 1 \\
\end{pmatrix}; \quad
\phi_\alpha = \frac12\text{ twist on }\omega = 
\begin{pmatrix}
1 & 0  \\
-1 & 1 \\
\end{pmatrix} ; \quad
\end{equation*}

and, up to sign, their composition 

\begin{equation*}
\phi_{\alpha}\phi_{\epsilon}=
\begin{pmatrix}
1 & 2  \\
-1 & -1 \\
\end{pmatrix}\:
\text{permutes the vectors}\:
\begin{pmatrix}
2  \\
 -1 \\
\end{pmatrix} \: \text{and} \:
\begin{pmatrix}
0  \\
 1 \\
\end{pmatrix} .)
\end{equation*}

    \begin{prop} \label{prop:-12} Consider $\kappa$ as shown in Figure \ref{fig:Kdef2}, and $p \in T^2 - \kappa$.  If either of the labeled meridian pairs $(a_0, b_1)$ or $(a_0, b_2)$ are standard in $T(p)$ then the element $-(12) \in \Dih_6$ is in $K(p)$.      
    \end{prop}
    
    \begin{proof}   As noted in Corollary \ref{cor:PpPq}, for a different choice $q \in T^2 - \kappa$, $K(p) = K(q)$, so it suffices to prove the Proposition for a point $p$ of our choice.  Let $p$ be a point lying in a bicollar $C$ of the curve $\bdd b_0$ but not on $\bdd b_0 \subset \kappa$ itself.  We have noted above that the element $-(12) \in \Dih_6$ corresponds to rotation $\psi = \rho_y$ around the $y$-axis, that is the axis which points into the page.  This rotation preserves the orientation of $b_0$ so $\psi^{-1}(p)$ and $p$ lie on the same side of $\bdd b_0$ in bicollar $C$.  In particular there is an arc $\gamma \subset T$ with ends at $\psi^{-1}(p)$ and $p$ that lies entirely in $C - \kappa$.    Then per Definition \ref{defin:thetap}, $\Theta_p(\psi)$ is represented by $\psi_{\psi^{-1}(p)} {\gamma}_p^{\psi^{-1}(p)}:(S^3, T(p) \to (S^3, T(p)$.  But away from $\gamma$, so in particular outside $C$, which contains $\frb_{\frc}$, this homeomorphism coincides with $\psi$ itself.  
    
    Now suppose $(a_0, b_1)$ is a standard meridian pair and compare Figure \ref{fig:Kdef2} and Figure \ref{fig:genswitch2}.  The figures coincide if we take $\alpha = \bdd a_0$ and $w_r = b_1$ so $\beta = \bdd b_0$.  The argument above shows that $\rho_y$ is the composition of, first, an eyeglass twist in which the bridge of the eyeglass intersects the corresponding standard bubble boundary exactly once, and, second, a standard flip.   The latter is a Powell move on $T(p)$ (as noted just before Proposition \ref{prop:newgen}), and so is the former, by Lemma \ref{lemma:eyeglass3o}.  Hence $\psi_{\psi^{-1}(p)} {\gamma}_p^{\psi^{-1}(p)}:(S^3, T(p)) \to (S^3, T(p))$ is a Powell move, so $\Theta_p(\psi) \in \calP(p)$, that is $\psi \in K(p)$ as required.    
    
   The same argument applies if $(a_0, b_2)$ is a standard meridian pair: just view Figure \ref{fig:Kdef2} from the other side, that is with the $y$-axis pointing out of the page.  Then $w_r = b_2$ and the same argument applies.  (The $\pi$-rotation around $\rho_y$ is then in the opposite direction, but as pointed out in the description of Figure \ref{fig:k23}, both directions represent the same element in $\calG_g$.)
\end{proof}

   \begin{prop}  \label{prop:+02} Consider $\kappa$ as shown in Figure \ref{fig:Kdef2}, and $p \in T^2 - \kappa$.  If either of the labeled meridian pairs $(a_0, b_1)$ or $(a_2, b_1)$ are standard in $T(p)$ then the element $+(02) \in \Dih_6$ is in $K(p)$.      
   \end{prop}
   
   \begin{proof}  The argument is symmetric to that of Proposition \ref{prop:-12}.  Suppose $(a_2, b_1)$ is standard and position so that in Figure \ref{fig:genswitch2} they appear as $(u_\ell, \beta)$.  Place $\frb_{\frc}$ in a bicollar of $\alpha = a_1$ (instead of $\beta = b_0$) and in the construction shown in Figure \ref{fig:genswitch3} do the flip on the bottom half of $P$ (instead of the top half), so that the orientation of $\alpha$ is preserved and that of $\beta$ is reversed.  In the planar viewpoint (Figure \ref{fig:genswitch2} with $\frb_{\frc}$ placed between $\alpha_{\pm}$ instead of $\beta_{\pm}$) this is equivalent to rotating around the $z$-axis, preserving the poles and both $a_1$ (orientation also preserved) and $b_1$ (orientation reversed).  Thus it represents $+(02)$, as required.  
   
   The case in which $(a_0, b_1)$ is standard is handled just as in the proof of Proposition \ref{prop:-12}: when viewed from the opposite side the pair appear as $(u_\ell, \beta)$ and the same argument as that for $(a_2, b_1)$ applies.  
   \end{proof}
    
 \begin{cor} \label{cor:Dih6inP} If any pair of labeled meridians are standard in $T(p)$, $K(p) = \Dih_6$.
     \end{cor}
    
\begin{proof} With no loss of generality we may as well assume that the pair $(a_0, b_1)$ is standard.  Following Propositions \ref{prop:-12} and \ref{prop:+02} we need only show that the elements $-(12), +(02)$ generate $\Dih_6$ and this is easy.  For example, since $[+(02)][-(12)] [+(02)] = -(01)$ and $[-(12)][+(02)][-(1,2)] = +(01)$ are in $K(p)$, the simple transposition of the poles $-\in K(p)$.  Then all transpositions $\pm(01), \pm(02), \pm(1,2) \in K(p)$ and these clearly generate $\Dih_6$.  
\end{proof}

\begin{proof}[Alternate proof of Proposition \ref{prop:ass1}]  We will apply Corollary \ref{cor:Dih6inP} in genus $g+1$.  The exchange of the standard bubble $(a_0, b_1)$, say, and the complementary bubble $(a_1, b_0)$ can be seen as an element $\psi \in \Dih_6$ by placing $a_0 = u_{\ell}$ and $b_1=w_{\ell}$ in Figure \ref{fig:genswitch2} and rotating around the $z$-axis.  The bubble $\frb_{\frc}$ is placed, as shown, at a fixed point of $\psi$.  By Corollary \ref{cor:Dih6inP} $\Theta_p(\psi) = \calP(p)$ and then Proposition \ref{prop:dihedral}(4) implies $\psi_p \in \calP(p)$ as required.  
\end{proof}

\end{document}